\documentclass[12pt]{amsart} 

\usepackage{amscd} 
\usepackage{amsmath} 
\usepackage{amssymb} 
\usepackage{amsthm}
\usepackage{bigdelim}
\usepackage{color} 
\usepackage{enumerate}
\usepackage{mathrsfs}
\usepackage{multirow}
\usepackage[all]{xy} 
\usepackage{hyperref}
\newtheorem{thm}{Theorem}[section] 
\newtheorem{cor}[thm]{Corollary}
\newtheorem{prop}[thm]{Proposition}
\newtheorem{conj}[thm]{Conjecture}

\newtheorem{lem}[thm]{Lemma}
\theoremstyle{definition} 
\newtheorem{defn}[thm]{Definition}

\newtheorem{eg}[thm]{Example} 
\theoremstyle{remark}
\newtheorem{rem}[thm]{Remark}

\newtheorem*{ack}{Acknowledgements}

\title{Numerical Kodaira dimension of algebraic fiber spaces in positive characteristic}
\author{Sho Ejiri}
\address{Department of Mathematics, Graduate School of Science, Osaka Metropolitan University, Osaka City, Osaka 558-8585, Japan}
\email{shoejiri.math@gmail.com}
\begin{document}
\maketitle
\markboth{SHO EJIRI}{Numerical Kodaira dimension of algebraic fiber spaces in positive characteristic}
\begin{abstract}
In this paper, we prove a positive characteristic analog of 
Nakayama's inequality on the numerical Kodaira dimension of 
algebraic fiber spaces when the generic fibers have nef canonical divisors.  
To this end, we establish variants of Popa and Schnell's 
global generation theorem, Viehweg's weak positivity theorem 
and Fujino's global generation theorem in positive characteristic. 
As a byproduct, we show that Iitaka's conjecture holds true 
in positive characteristic when the base space is of general type 
and the canonical divisor of the total space is relatively semi-ample.
\end{abstract}
\section{Introduction}
We first recall Iitaka's conjecture on the Kodaira dimension of algebraic fiber spaces:
\begin{conj}[\textup{\cite{Iit72}}] \label{conj:Iitaka}
Let the base field be an algebraically closed field of characteristic zero. 
Let $f:X\to Y$ be a surjective morphism with connected fibers between 
smooth projective varieties. 
Then the inequality 
\begin{align}
\kappa(X) \ge \kappa(X_{\overline\eta}) + \kappa(Y)
\tag{I} \label{ineq:Iitaka}
\end{align}
holds, where $X_{\overline\eta}$ is the geometric generic fiber of $f$. 
\end{conj}
In \cite{Nak04}, Nakayama introduced the notion of numerical Kodaira dimension
$\kappa_\sigma$, and proved the following inequalities:
\begin{thm}[\textup{Special case of \cite[V.~4.1.~Theorem~(1)]{Nak04}, see also \cite[\S 3]{Fuj20}}]  \samepage
\label{thm:Nakayama_intro}
Let $f:X\to Y$ be as in Conjecture~\ref{conj:Iitaka}. 
Let $D$ be a $\mathbb Q$-divisor on $X$ such that $D-K_{X/Y}$ is nef. 
Then for every $\mathbb Q$-divisor $Q$ on $Y$, 
$$
\kappa_\sigma(D+f^*Q) 
\ge \kappa_\sigma\left(D|_{X_{\overline\eta}}\right) +\kappa(Q)
$$
and 
$$
\kappa_\sigma(D+f^*Q) 
\ge \kappa\left(D|_{X_{\overline\eta}}\right) +\kappa_\sigma(Q). 
$$
In particular, 
$$
\kappa_\sigma(X) \ge \kappa_\sigma\left(X_{\overline\eta}\right) +\kappa(Y)
$$
and 
$$
\kappa_\sigma(X) \ge \kappa\left(X_{\overline\eta}\right) +\kappa_\sigma(Y). 
$$
\end{thm}
Here, $\kappa$ stands for the Iitaka dimension of divisors. 
The numerical Kodaira dimension is conjectured to be 
equal to the Kodaira dimension (cf. \cite[Conjecture~1.4]{Fuj17s}), 
so up to this conjecture, Theorem~\ref{thm:Nakayama_intro} solves 
affirmatively Conjecture~\ref{conj:Iitaka}. 

In this paper, we prove a positive characteristic analog of 
Theorem~\ref{thm:Nakayama_intro} when the geometric generic fiber has 
nef canonical divisor. 
From here on, let the base field be an $F$-finite field $k$ 
of characteristic $p>0$ (i.e. $[k:k^p]<\infty$). 
Let $f:X\to Y$ be a surjective morphism from a regular projective 
variety $X$ to a regular projective variety $Y$ of dimension $n$. 
Let $X_{\overline\eta}$ denote the geometric generic fiber of $f$. 
Then the main theorem of this paper can be stated as follows:
\begin{thm}[\textup{Theorem~\ref{thm:subadditivity}}] \samepage
\label{thm:subadditivity_intro}
Suppose that $X_{\overline\eta}$ is smooth and $K_{X_{\overline\eta}}$ is nef. 
Let $D$ be a $\mathbb Q$-divisor on $X$ such that $D-K_{X/Y}$ is nef. 
Then for every $\mathbb Q$-divisor $Q$ on $Y$, 
$$
\kappa_\sigma(D+f^*Q) \ge \kappa_\sigma\left(D|_{X_{\overline\eta}}\right) 
+\kappa(Q)
$$
and 
$$
\kappa_\sigma(D+f^*Q) \ge \kappa\left(D|_{\overline\eta}\right)
+\kappa_\sigma(Q). 
$$
\end{thm}
In the case when $X_{\overline\eta}$ has ``bad'' singularities, there exist counterexamples to Theorem~\ref{thm:subadditivity_intro} (see Example~\ref{eg:CEKZ}). 

Since the numerical dimension of a semi-ample divisor coincides with 
the Iitaka dimension, we obtain the following as a corollary of 
Theorem~\ref{thm:subadditivity_intro}: 
\begin{cor} \label{cor:Iitaka-semiample_intro}
Suppose that $X_{\overline\eta}$ is smooth and $K_X$ is semi-ample. 
Then inequality~$(\ref{ineq:Iitaka})$ holds. 
\end{cor} 
To prove Theorem~\ref{thm:subadditivity_intro}, we show variants of Popa and Schnell’s global generation theorem (\cite[Theorem 1.4]{PS14}, see also \cite[Theorem 1.3]{Eji22a}), Viehweg’s weak positivity theorem (\cite[Folgerung 3.4]{Vie82}) and Fujino’s global generation theorem (\cite[Theorem 1.7]{Fuj19}) in positive characteristic. As a byproduct of these variants, we obtain the following theorem:
\begin{thm}[\textup{Theorem~\ref{thm:relatively_semi-ample}}] \label{thm:relatively_semi-ample_intro}
Suppose that $X_{\overline\eta}$ is smooth, 
$K_X$ is relatively semi-ample over $Y$, and $Y$ is of general type. 
Then inequality (\ref{ineq:Iitaka}) holds true. 
\end{thm}
Let us move on to the explanation on the variants mentioned above. 
We first prove a variant of Popa and Schnell’s global generation theorem in positive characteristic:
\begin{thm}[\textup{Theorem~\ref{thm:PS-type_generic}}]  \samepage
\label{thm:PS-type_generic_intro}
Let $A$ be a sufficiently ample divisor on $X$. 
Let $\mathcal L$ be an ample line bundle on $Y$ 
and let $j$ be the smallest positive integer such that $|\mathcal L^j|$ is free. 
Suppose that $K_{X_{\overline\eta}}$ is nef. 
Then the sheaf 
$$
f_*\mathcal O_X(mK_X+A+N) \otimes \mathcal L^l
$$
is generically generated by its global sections 
for each $m\ge 1$, $l\ge m(jn+1)$ and every nef divisor $N$ on $X$. 
\end{thm}
Here, ``sufficiently ample divisor'' is, for example, $A=iA'$ for 
an integer $i$ large enough and an ample divisor $A'$ on $X$. 

In characteristic zero, Theorem~\ref{thm:PS-type_generic_intro} 
follows from the log version of 
Popa and Schnell's global generation theorem~(\cite[Variant~1.6]{PS14}). 
Indeed, take a smooth member $H\in|A+N|$ and put $\Delta:=m^{-1}H$. 
Then 
$$
f_*\mathcal O_X(mK_X+A+N) \otimes \mathcal L^l
\cong f_*\mathcal O_X(m(K_X+\Delta)) \otimes \mathcal L^l
$$
is generated by its global sections for $l\ge m(n+1)$ (when $j=1$). 
 
In positive characteristic, however, we cannot apply the same argument as above, 
though an analog of the log version of Popa and Schnell's 
global generation theorem has been established 
(\cite[Theorem~6.11]{Eji19d}, see also \cite[Theorem~1.5]{Eji22a}). 
This is because $m$ in \cite[Theorem~6.11]{Eji19d} is required to be 
larger than $m_0=m_0(f,\Delta)$ that depends on the morphism $f$ and the boundary divisor $\Delta$. 

In the case when the canonical divisor of the total space is relatively nef, 
we have the following theorem: 
\begin{thm}[\textup{Theorem~\ref{thm:PS-type}}] \label{thm:PS-type_intro} \samepage
Let $A$, $\mathcal L$ and $j$ be as in Theorem~\ref{thm:PS-type_generic_intro}. 
Suppose that $K_X$ is $f$-nef. 
Then the sheaf 
$$
f_*\mathcal O_X(mK_X+A+N) \otimes \mathcal L^l
$$
is generated by its global sections for each $m\ge 1$, $l\ge m(jn+1)$ and 
every nef divisor $N$ on $X$. 
\end{thm}
Using Theorem~\ref{thm:PS-type_generic_intro}, we prove
a variant of Viehweg's weak positivity theorem in positive characteristic: 
\begin{thm}[\textup{Theorem~\ref{thm:relative_generic}}]  \samepage
\label{thm:relative_intro}
Let $A$ be a sufficiently ample divisor on $X$. 
Suppose that $X_{\overline\eta}$ is smooth and $K_{X_{\overline\eta}}$ is nef. 
Then for each $m\ge 1$ and every nef divisor $N$ on $X$, the sheaf 
$$
f_*\mathcal O_X(mK_{X/Y}+A+N)
$$
is weakly positive in the sense of \cite{Vie83}. 
\end{thm}
We apply the above theorem to show a variant of Fujino's 
global generation theorem in positive characteristic, which is used to prove 
Theorem~\ref{thm:subadditivity_intro}. 
\begin{thm}[\textup{Theorem~\ref{thm:relative2_generic}}]  \samepage
\label{thm:relative2_intro}
Let $A$ be a sufficiently ample divisor on $X$. 
Suppose that $X_{\overline\eta}$ is smooth and $K_{X_{\overline\eta}}$ is nef. 
Let $\mathcal L$ and $\mathcal H$ be ample line bundles with 
$|\mathcal L|$ free. 
Then for each $m\ge 1$ and every nef divisor $N$ on $X$, the sheaf 
$$
f_*\mathcal O_X(mK_{X/Y} +A+N) \otimes \omega_Y 
\otimes \mathcal L^n \otimes \mathcal H
$$
is generically generated by its global sections. 
\end{thm}
This paper is organized as follows: 
In Section~\ref{section:defn}, 
we define several notions and notations used in this paper. 
In Section~\ref{section:nef}, we treat algebraic fiber spaces 
whose generic fibers have nef canonical divisors, and prove several 
global generation theorem and the weak positivity theorem. 
Section~\ref{section:subadditivity} is devoted to prove the main theorem 
of this paper (Theorem~\ref{thm:subadditivity}). 
In Section~\ref{section:relatively_semi-ample}, we deal with algebraic fiber spaces whose total spaces have 
relatively semi-ample canonical divisors, and show that Iitaka’s conjecture 
holds true if the base spaces are of general type.
In Section~\ref{section:mAd}, we consider algebraic fiber spaces over 
varieties of maximal Albanese dimension, without assuming that 
the geometric generic fiber is $F$-pure (Definition~\ref{defn:F-pure}), 
and show that a couple of inequalities on the numerical dimension 
similar to Theorem~\ref{thm:subadditivity_intro} 
(Theorem~\ref{thm:subadditivity_mAd}). 
\begin{ack}
The author would like to thank Professors 
Osamu Fujino, Noboru Nakayama and Shunsuke Takagi 
for helpful and valuable comments. 
This work was partly supported by MEXT Promotion of Distinctive Joint Research Center Program JPMXP0619217849.
\end{ack}
\section{Definitions and Terminologies} \label{section:defn}
In this section, we define several notions and notations used in this paper. 

First, we define some terminologies. 

Throughout this paper, we work over an $F$-finite field $k$ of 
characteristic $p>0$ (i.e. $[k:k^p] <\infty$). 
A \textit{$k$-scheme} is a separated scheme of finite type over $k$. 
A \textit{variety} is an integral $k$-scheme. 

Let $X$ be an equi-dimensional $k$-scheme satisfying $S_2$ and $G_1$. 
Here, $S_2$ (resp. $G_1$) means Serre's second condition (resp. the condition 
that it is Gorenstein in codimension one). 
Let $\mathcal K$ be the sheaf of total quotient rings on $X$. 
An \textit{AC divisor} (or \textit{almost Cartier divisor}) 
is a coherent subsheaf of $\mathcal K$ that is invertible 
in codimension one (\cite[p. 301]{Har94}, \cite[Definition~2.1]{MS12}). 
Let $D$ be an AC divisor on $X$. 
Let $\mathcal O_X(D)$ denote the coherent sheaf defining $D$. 
We say that $D$ is \textit{effective} if $\mathcal O_X\subseteq\mathcal O_X(D)$. 
In this paper, we say that an effective divisor is \textit{prime}
if it cannot be written as the sum of two non-zero effective divisors. 
The set $\mathrm{WSh}(X)$ of AC divisors forms an additive group 
\cite[Corollary~2.6]{Har94}. 
Let $\mathbb Z_{(p)}$ denote the localization of the ring of integers $\mathbb Z$
at the prime ideal $(p)=p\mathbb Z$. 
A \textit{$\mathbb Z_{(p)}$-AC divisor} is an element of 
$\mathrm{WSh}(X)\otimes_{\mathbb Z}\mathbb Z_{(p)}$. 
Let $\Delta$ be a $\mathbb Z_{(p)}$-AC divisor on $X$. 
Then $\Delta$ can be written as a linear combination $\sum_i\delta_i\Delta_i$, 
where $\delta_i\in \mathbb Z_{(p)}$ and each $\Delta_i$ is prime. 
This decomposition is not necessarily unique (cf. \cite[(16.1.2)]{Kol+14}). 
\textbf{In this paper, we fix a decomposition $\sum_i\delta_i\Delta_i$ into prime divisors for a given $\mathbb Z_{(p)}$-AC divisor $\Delta$.} 
By this rule, for an integer $m\ge1$, the AC divisor $m\Delta$ 
is well-defined when each $m\delta_i$ is an integer. 
We say that a $\mathbb Z_{(p)}$-AC divisor $\Delta$ is 
\textit{$\mathbb Z_{(p)}$-Cartier} if $m\Delta$ is Cartier for an integer $m>0$
with $p\nmid m$. 

Let $\varphi:S\to T$ be a morphism of $k$-schemes. 
Let $T'$ be another $T$-scheme. 
Then we let $S_{T'}$ denote the fiber product $S\times_T T'$. 
The second projection $S\times_T T'\to T'$ is denoted by 
$\varphi_{T'}:S_{T'}\to T'$. 
Let $D$ be a Cartier divisor (resp. AC divisor, $\mathbb Z_{(p)}$-AC divisor) 
on $S$. 
Let $D_{T'}$ denote the pullback of $D$ to $S_{T'}$, if it is well-defined. 
Similarly, for a coherent sheaf $\mathcal F$ on $S$, 
we let $\mathcal F_{T'}$ be the pullback of $\mathcal F$ to $S_{T'}$. 

Let $X$ be a $k$-scheme. Let $e\ge 1$ be an integer. 
We denote by $F_X^e:X\to X$ the $e$-times iterated absolute Frobenius morphism on $X$. 
When we regard $X$ as the source of $F_X^e$, we denote it by $X^e$. 
Let $f:X\to Y$ be a morphism of $k$-schemes. 
We denote $f$ by $f^{(e)}$ if we regard $X$ and $Y$ as the sources of 
$F_X^e$ and $F_Y^e$, respectively; $f^{(e)}:X^e\to Y^e$. 
Then we can define the $e$-th relative Frobenius morphism 
$F_{X/Y}^{(e)}:=(F_X^e, f^{(e)}): X^e \to X_{Y^e}=X\times_Y Y^e$. 
We have the following commutative diagram: 
$$
\xymatrix{
X^e \ar[d]_-{F_{X/Y}^{(e)}} \ar[dr]^-{F_X^e} \ar@/_50pt/[dd]_-{f^{(e)}} & \\ 
X_{Y^e} \ar[r] \ar[d]_-{f_{Y^e}} & X \ar[d]^-f \\ 
Y^e \ar[r]_-{F_Y^e} & Y
}
$$

Next, we collect several definitions used in this paper. 
\begin{defn}[\textup{\cite[Definition 1]{Les22}, \cite[V.2.5]{Nak04}}] \label{defn:num dim}
Let $X$ be a normal projective variety. 
Let $D$ be an $\mathbb Q$-Cartier $\mathbb Q$-divisor on $X$. 
The \textit{numerical dimension} $\kappa_\sigma^{\mathbb R}(D)$ 
(resp. $\kappa_\sigma(D)$) of $D$ 
is the supremum (resp. maximum) of real numbers (resp. non-negative integers) 
$r$ such that 
$$
\underset{m\to\infty}{\mathrm{lim~sup}}~
\frac{h^0(X, \lfloor mD \rfloor +A)}{m^r}
>0
$$
for some ample Cartier divisor $A$ on $X$. 
If there exists no such $r$, we define 
$\kappa_\sigma^{\mathbb R}(D)=\kappa_{\sigma}(D)=-\infty$. 

When $X$ is regular, we define the \textit{numerical Kodaira dimension} 
$\kappa_\sigma(X)$ by 
$$
\kappa_\sigma(X):=\kappa_\sigma(K_X).
$$ 
\end{defn}
\begin{defn} \label{defn:positivity}
Let $Y$ be a quasi-projective $k$-scheme 
and $\mathcal G$ be a coherent sheaf on $Y$. 
Let $V$ be an open subset of $Y$. 
\\ \noindent(1) 
We say that $\mathcal G$ \textit{is generated by its global sections over $V$} or \textit{is globally generated over $V$} if the natural morphism 
$
H^0(Y,\mathcal G) \otimes_k \mathcal O_Y 
\to \mathcal G
$
is surjective over $V$. 
\\ \noindent(2) 
We say that $\mathcal G$ \textit{is generically generated by its global sections} or \textit{is generically globally generated} if the natural morphism 
$
H^0(Y,\mathcal G) \otimes_k \mathcal O_Y 
\to \mathcal G
$
is generically surjective. 
\\ \noindent(3)~(\cite[Variant~2.13]{Vie95}) 
Assume that $Y$ is normal. 
Let $\mathcal T$ be the torsion part of $\mathcal G$ 
and set $\mathcal G':=\mathcal G/\mathcal T$. 
Suppose that $\mathcal G|_V$ is locally free. 
Let $\mathcal H$ be an ample line bundle on $Y$. 
We say that $\mathcal G$ is \textit{weakly positive over $V$}
if any integer $\alpha\ge1$, 
there exists an integer $\beta\ge 1$ such that the sheaf 
$$
(S^{\alpha\beta}(\mathcal G'))^{\vee\vee} \otimes \mathcal H^{\beta}
$$
is generated by its global sections over $V$. 
Here, $S^m(?)$ denotes the $m$-th symmetric product and $(?)^{\vee\vee}$ denotes
the double dual. 
\\ \noindent(4) 
Assume that $Y$ is normal. 
Let $\mathcal G'$ and $\mathcal H$ be as in~(3). 
We say that $\mathcal G$ is \textit{pseudo-effective} if 
for any integer $\beta\ge1$, 
there exists an integer $\beta\ge 1$ such that the sheaf 
$$
(S^{\alpha\beta}(\mathcal G'))^{\vee\vee} \otimes \mathcal H^{\beta}
$$
is generically generated by its global sections. 
\end{defn}
\begin{rem} \label{rem:weak pos}
The notions of weak positivity and pseudo-effectivity are 
independent of the choice of the ample line bundle $\mathcal H$ 
(\cite[Lemma~2.14~a)]{Vie95}). 
\end{rem}
\begin{defn} \label{defn:F-pure}
Let $X$ be an equi-dimensional $k$-scheme satisfying $S_2$ and $G_1$. 
Let $\Delta$ be an effective $\mathbb Z_{(p)}$-AC divisor on $X$. 
We say that the pair $(X,\Delta)$ is \textit{$F$-pure} 
if for any integer $e\ge1$ such that $(p^e-1)\Delta$ is integral, 
the composite of 
$$
\mathcal O_X
\xrightarrow{{F_X^e}^\#} {F_X^e}_*\mathcal O_X
\hookrightarrow {F_X^e}_* \mathcal O_X((p^e-1)\Delta)
$$
locally splits as an $\mathcal O_X$-module homomorphism. 
\end{defn}
When $X$ is a normal variety, the above definition is equivalent to that in~\cite{HW02}. 
 
\medskip

Next, we define some notations used in this paper. 
Let $X$ be an equi-dimensional $k$-scheme satisfying $S_2$ and $G_1$. 
Let $\Delta$ be an effective $\mathbb Z_{(p)}$-AC divisor on $X$. 
Let $e_0$ be the smallest positive integer such that $(p^{e_0}-1)\Delta$ 
is integral. Then $(p^e-1)\Delta$ is integral for each $e\ge 1$ with $e_0|e$. 
Let us consider the composite 
$$
\mathcal O_X 
\xrightarrow{{F_X^e}^\#} {F_X^e}_*\mathcal O_X 
\hookrightarrow {F_X^e}_*\mathcal O_X((p^e-1)\Delta)
$$
for an integer $e\ge 1$ with $e_0|e$. 
Let $M$ be a Cartier divisor on $X$. 
Applying $\mathcal Hom(?, \mathcal O_X(M))$ to the above composite, 
we obtain 
$$
{F_X^e}_*\mathcal O_X((1-p^e)(K_X+\Delta)+p^eM)
\to \mathcal O_X(M)
$$
by the Grothendieck duality and the projection formula. 
Let $\phi_{(X,\Delta)}^{(e)}(M)$ denote this morphism. 
Let $f:X\to Y$ be a morphism to a regular variety $Y$. 
Consider the composite 
$$
\mathcal O_{X_{Y^e}}
\xrightarrow{{F_{X/Y}^{(e)}}^\#} {F_{X/Y}^{(e)}}_*\mathcal O_X
\hookrightarrow {F_{X/Y}^{(e)}}_*\mathcal O_X((p^e-1)\Delta)
$$
for each $e\ge1$ with $e_0|e$. 
Applying $\mathcal Hom(?,\mathcal O_{X_{Y^e}}(M_{Y^e}))$ to the above composite, 
we get 
$$
{F_{X/Y}^{(e)}}_*\mathcal O_X((1-p^e)(K_{X/Y}+\Delta) +p^eM)
\to \mathcal O_{X_{Y^e}}(M_{Y^e}). 
$$
Let $\phi_{(X/Y,\Delta)}^{(e)}(M_{Y^e})$ denote this morphism. 
\section{Algebraic fiber spaces whose generic fibers have nef canonical divisors}
\label{section:nef}
In this section, we treat an algebraic fiber space whose generic fiber 
has nef canonical divisor. 
To prove Theorem~\ref{thm:PS-type}, we need the following lemmas. 
\begin{lem} \label{lem:rel free}
Let $f:X\to Y$ be a projective morphism between $k$-schemes. 
Let $A$ be an $f$-ample Cartier divisor on $X$. 
Then there exists an integer $m_0 \ge 1$ such that 
$\mathcal O_X(mA+N)$ is relatively free over $Y$ 
for each $m\ge m_0$ and every $f$-nef Cartier divisor $N$. 
\end{lem}
\begin{proof}
This follows from the relative Castelnuovo--Mumford regularity \cite[Example~1.8.24]{Laz04I} and Keeler's relative Fujita vanishing theorem \cite[Theorem~1.5]{Kee03}. 
\end{proof}
\begin{lem}[\textup{\cite[Lemma~3.4]{Eji19p}}] \label{lem:glgen}
Let $f:X\to Y$ be a morphism between projective $k$-schemes. 
Let $\mathcal F$ be a coherent sheaf on $X$ 
and let $A$ be an ample Cartier divisor on $X$. 
Then there exists an integer $m_0\ge 1$ such that 
$$
f_* \mathcal F(mA+N)
$$
is generated by its global sections for each $m\ge m_0$ 
and every nef Cartier divisor $N$ on $X$. 
\end{lem}
\begin{lem} \label{lem:Frobenius glgen}
Let $W$ be a projective variety of dimension $n$, 
let $L$ be an ample line bundle on $W$, 
and let $j$ be the smallest positive integer such that $|jL|$ is free. 
Let $Y$ be an open subset of $W$ and let $\mathcal G$ be a coherent sheaf on $Y$. 
Let $\{a_e\}_{e\ge 1}$ be a sequence of positive integers such that 
$a_e/p^e$ converges to $\varepsilon +jn$ for an $\varepsilon\in\mathbb R_{>0}$. 
Then there exists an $e_0\ge 1$ such that 
$$
{F_Y^e}_*\mathcal G(a_eL|_Y)
$$
is generated by its global sections for each $e\ge e_0$. 
\end{lem}
\begin{proof}
Let $\mathcal G'$ be a coherent sheaf on $W$ 
such that $\mathcal G'|_Y\cong \mathcal G$. 
Then we have 
$$
\left({F_W^e}_*\mathcal G'(a_eL)\right)|_Y \cong {F_Y^e}_* \mathcal G(a_eL|_Y), 
$$
so we may assume that $W=Y$. 
We show that ${F_Y^e}_*\mathcal G(a_eL)$ is $0$-regular with respect to $jL$ 
in the sense of Castelnuovo--Mumford. 
If this holds, then our claim follows from \cite[Theorem~1.8.5]{Laz04I}. 
For each $0<i\le n$, we have 
\begin{align*}
H^i\big(Y, \left( {F_Y^e}_*\mathcal G(a_eL) \right) 
\otimes \mathcal O_Y(-ijL)\big) 
& \cong H^i\big(Y, {F_Y^e}_*\mathcal G((a_e-p^eij)L) \big) 
\\ & \cong H^i\big(Y, \mathcal G((a_e-p^eij)L) \big) 
=:V_{i,e}
\end{align*}
by the projection formula. 
Since $a_e/p^e -ij \xrightarrow{e\to \infty} \varepsilon +(n-i)j >0$, 
we see that $a_e-p^eij\xrightarrow{e\to \infty} \infty$, 
so by the Serre vanishing, there is an $e_0\ge 1$ such that 
$V_{i,e}=0$ for each $e\ge e_0$ and $0<i\le n$. 
\end{proof}
\begin{thm} \label{thm:PS-type} \samepage
Let $\bar f:\bar X \to \bar Y$ be a surjective morphism 
from a projective equi-dimensional $k$-scheme $\bar X$ 
to a projective variety $\bar Y$ of dimension $n$. 
Let $Y$ be a dense open subset of $\bar Y$ 
and set $X:=\bar f^{-1}(Y)$ and $f:=\bar f|_X:X\to Y$. 
Let $\bar A$ be a sufficiently ample Cartier divisor on $\bar X$ 
and put $A:=\bar A|_X$. 
Let $\bar{\mathcal L}$ be an ample line bundle on $\bar Y$ and let 
$j$ be the smallest positive integer such that $|\bar{\mathcal L}^j|$ is free. 
Set $\mathcal L:=\bar{\mathcal L}|_Y$. 
Assume that $X$ satisfies $S_2$ and $G_1$. 
Let $\Delta$ be an effective $\mathbb Z_{(p)}$-AC divisor on $X$. 
Suppose further that 
\begin{itemize}
\item $K_X+\Delta$ is $\mathbb Z_{(p)}$-Cartier, 
\item there exists an open subset $V$ of $Y$ such that 
	$(U,\Delta|_U)$ is $F$-pure, where $U:=f^{-1}(V)$, and 
\item $K_U+\Delta|_U$ is nef over $V$. 
\end{itemize}
Let $\bar N$ be a Cartier divisor on $\bar X$ such that 
$\bar A +\bar N$ is ample and $\bar N|_U$ is nef over $V$ $($e.g. $N$ is nef$)$. 
Put $N:=\bar N|_X$. 
Let $m\ge1$ be an integer such that $m(K_X+\Delta)$ is Cartier. 
Then the sheaf 
$$
f_*\mathcal O_X(m(K_X+\Delta) +A +N)
\otimes \mathcal L^l 
$$
is generated by its global sections over $V$ for each $l\ge m(jn+1)$. 
\end{thm}
We have the following commutative diagram whose each square is cartesian: 
$$
\xymatrix{
U \ar[d]_-g \ar@{^(->}[r] 
& X \ar[d]^-f \ar@{^(->}[r] 
& \bar X \ar[d]^-{\bar f} \\
V \ar@{^(->}[r]
& Y \ar@{^(->}[r] 
& \bar Y. 
}
$$
Here, $g:=f|_U:U\to V$. 
\begin{proof}
The proof is partially based on the idea of the referee of \cite{Eji19d}. 
For simplicity, let $K$ denote $K_X+\Delta$. 
Let $e_0\ge 1$ be an integer such that $(p^{e_0}-1)K$ is Cartier. 
Then $(p^e-1)K$ is Cartier for each $e\ge 1$ with $e_0|e$. 
Since $\bar A$ is sufficiently ample, the following hold: 
\begin{enumerate}
\item $\mathcal O_U((mK +A +N)|_U)$ is $g$-ample and relatively free over $V$ by Lemma~\ref{lem:rel free}. 
\item The morphism 
\begin{align*}
f_*\phi^{(e)}_{(X,\Delta)}(mK+A+N): 
& {F_Y^e}_*f_*\mathcal O_X\big(((m-1)p^e+1)K +p^e(A+N) \big) 
\\ & \to f_* \mathcal O_X(mK+A+N)
\end{align*}
is surjective over $V$. 
This follows from an argument similar to that of the proof of \cite[Corollary~2.23]{Pat14}. Note that $K|_U$ is $g$-nef and $(U,\Delta|_U)$ is $F$-pure. 
\item 
For every Cartier divisor $M$ on $U$, 
the sheaf $\mathcal O_U(M+l(A+N)|_U)$ is $0$-regular with respect to 
$\mathcal O_U((mK +A+N)|_U)$ and $g$ (\cite[Example~1.8.24]{Laz04I})
for $l\gg0$. 
\end{enumerate}
Let $q_e$ and $r_e$ are the quotient and the remainder of the division of $(m-1)p^e+1$ by $m$. 
Then by (3), we see that the natural morphism 
\begin{align*}
S^{q_e}(f_*\mathcal O_X(mK+A+N))
& \otimes 
f_*\mathcal O_X(r_eK +(p^e-q_e)(A+N))
\\ & \to f_*\mathcal O_X\big(((m-1)p^e+1)K +p^e(A+N) \big)
\end{align*}
is surjective over $V$ for each $e\gg1$ with $e_0|e$. 
Note that $r_eK$ is Cartier and $p^e-q_e \xrightarrow{e\to \infty} \infty$. 
Let $\mathcal F_{r_e}$ be the coherent sheaf on $\bar X$ 
such that $\mathcal F_{r_e}|_X \cong \mathcal O_X(r_eK)$. 
By Lemma~\ref{lem:glgen} and the ampleness of $\bar A+\bar N$, 
we see that 
$$
\bar f_*\mathcal F_{r_e}\left((p^e-q_e)
\left(\bar A+\bar N\right)\right) |_Y
\cong f_*\mathcal O_X(r_eK +(p^e-q_e)(A+N))
$$
is globally generated for $e\gg0$. 
Hence, we obtain the morphisms
\begin{align*}
& \bigoplus S^{q_e}(f_*\mathcal O_X(mK+A+N))
\\ & \cong S^{q_e}(f_*\mathcal O_X(mK+A+N)) 
\otimes \left(H^0(X, r_eK+(p^e-q_e)(A+N)) \otimes_k \mathcal O_Y \right)
\\ & \twoheadrightarrow S^{q_e}(f_*\mathcal O_X(mK+A+N))
\otimes f_*\mathcal O_X(r_eK +(p^e-q_e)(A+N))
\\ & \to f_*\mathcal O_X\big(((m-1)p^e+1)K +p^e(A+N) \big) 
\end{align*}
that are surjective over $V$. 
Let $\alpha^{(e)}$ denote the composite of these morphisms. 
Combining ${F_Y^e}_*\alpha^{(e)}$ and $f_*\phi^{(e)}_{(X,\Delta)}(mK+A+N)$, 
we get the morphism 
\begin{align*}
\bigoplus {F_Y^e}_*S^{q_e}(f_*\mathcal O_X(mK+A+N))
\to 
f_*\mathcal O_X(mK+A+N)
\end{align*}
that is surjective over $V$, 
for each $e\ge 1$ with $e_0|e$. 
Taking the tensor product with $\mathcal O_Y(lL)$, 
where $L$ is a Cartier divisor on $Y$ with $\mathcal O_Y(L)\cong \mathcal L$, 
we obtain 
$$
\bigoplus {F_Y^e}_* \big( S^{q_e}(f_*\mathcal O_X(mK+A+N)) (p^elL) \big)
\to 
f_*\mathcal O_X(mK+A+N) (lL)
$$
by the projection formula. 
Let $s$ be the smallest integer such that 
$$
f_*\mathcal O_X(mK+A+N)(sL)
$$
is globally generated over $V$. 
Then we have the following sequence of morphisms that are surjective over $V$:
\begin{align*}
& \bigoplus {F_Y^e}_* \mathcal O_Y((p^el -q_es)L)
\\ & \to \bigoplus {F_Y^e}_* \big( 
S^{q_e}(f_*\mathcal O_X(mK+A+N)(sL)) ((p^el-q_es)L) \big)
\\ & \cong \bigoplus {F_Y^e}_* \big( 
S^{q_e}(f_*\mathcal O_X(mK+A+N)) (p^elL) \big)
\\ & \to f_*\mathcal O_X(mK+A+N) (lL). 
\end{align*}
Now, we have 
$$
\frac{p^el-q_es}{p^e} 
\xrightarrow{e\to \infty} 
l -\frac{m-1}{m}s, 
$$
so if $l > \frac{m-1}{m}s +jn$, then ${F_Y^e}_*\mathcal O_Y((p^el-q_es)L)$ is 
globally generated by Lemma~\ref{lem:Frobenius glgen}, and hence 
$f_*\mathcal O_X(mK+A+N)(lL)$ is globally generated over $V$ 
by the above morphisms.  
Then by the choice of $s$, we get 
$$
s \le \frac{m-1}{m}s +jn +1, 
$$
which means that $s \le m(jn+1)$. 
Hence, if $l\ge m(jn+1)$, then 
$$
\frac{m-1}{m}s +jn 
\le \frac{m-1}{m}m(jn+1) +jn 
=(m-1)(jn+1) +jn
< m(jn+1) \le l, 
$$
so $f_*\mathcal O_X(mK+A+N)(lL)$ is globally generated over $V$, 
which is our assertion. 
\end{proof}
When the generic fiber has nef canonical divisor, 
by an argument similar to the above, we can prove the following theorem: 
\begin{thm} \label{thm:PS-type_generic} \samepage
Let $\bar f:\bar X \to \bar Y$, $f:X\to Y$, $\bar A$, $A$, $\bar{\mathcal L}$, 
$\mathcal L$ and $j$ be as in Theorem~\ref{thm:PS-type}. 
Assume that $X$ satisfies $S_2$ and $G_1$. 
Let $\Delta$ be an effective $\mathbb Z_{(p)}$-AC divisor on $X$. 
Suppose further that 
\begin{itemize}
\item $K_X+\Delta$ is $\mathbb Z_{(p)}$-Cartier, 
\item $(X_\eta,\Delta|_{X_\eta})$ is $F$-pure, 
	where $X_\eta$ is the generic fiber of $f$, and
\item $(K_X+\Delta)|_{X_\eta}$ is nef. 
\end{itemize}
Let $\bar N$ be a Cartier divisor on $\bar X$ such that 
$\bar A +\bar N$ is ample and $\bar N|_{X_\eta}$ is nef $($e.g. $N$ is nef$)$. 
Put $N:=\bar N|_X$. 
Let $m\ge1$ be an integer such that $m(K_X+\Delta)$ is Cartier. 
Then the sheaf 
$$
f_*\mathcal O_X(m(K_X+\Delta) +A +N) \otimes \mathcal L^l 
$$
is generically generated by its global sections for each $l\ge m(jn+1)$. 
\end{thm}
\begin{proof}
For simplicity, let $K$ denote $K_X+\Delta$. 
Let $e_0\ge 1$ be an integer such that $(p^{e_0}-1)K$ is Cartier. 
Since $\bar A$ is sufficiently ample, the following hold:
\begin{enumerate}
\item $\mathcal O_{X_{\overline\eta}}\left((mK+A+N)|_{X_{\overline\eta}}\right)$
is ample and its complete linear system is free by Lemma~\ref{lem:rel free}. 
\item The morphism 
\begin{align*}
f_*\phi^{(e)}_{(X,\Delta)}(mK+A+N): 
& {F_Y^e}_*f_*\mathcal O_X\big(((m-1)p^e+1)K +p^e(A+N) \big) 
\\ & \to f_* \mathcal O_X(mK+A+N)
\end{align*}
is generically surjective for each $e\ge1$ with $e_0|e$. 
\item 
For every Cartier divisor $M$ on $X_{\overline\eta}$, 
the sheaf 
$
\mathcal O_{X_{\overline\eta}}\left(M+l(A+N)|_{X_{\overline\eta}}\right)
$ 
is $0$-regular with respect to 
$\mathcal O_U\left((mK +A+N)|_{X_{\overline\eta}}\right)$ 
for $l\gg0$. 
\end{enumerate}
From here on, we can prove the assertion, applying the same argument as that 
of the proof of Theorem~\ref{thm:PS-type}. 
Note that the morphisms used in the proof of Theorem~\ref{thm:PS-type} 
are generically smooth in this case. 
\end{proof}
Next, we prove Theorem~\ref{thm:relative}. 
To this end, we need the following proposition: 
\begin{prop} \label{prop:e-surjective}
Let $f:X\to Y$ be a flat projective surjective morphism 
from an equi-dimensional $k$-scheme $X$ satisfying $S_2$ and $G_1$ 
to a regular variety $Y$. 
Let $\Delta$ be an effective $\mathbb Z_{(p)}$-AC divisor on $X$. 
Let $W$ be a largest Gorenstein open subset of $X$. 
Suppose that 
\begin{itemize}
\item $K_X+\Delta$ is $\mathbb Z_{(p)}$-Cartier, 
\item the codimension of $(X\setminus W)|_{X_y}$ $($resp. $\mathrm{Supp}(\Delta)|_{X_y})$ in $X_y$ is at least two $($resp. one$)$ for every $y\in Y$, and 
\item $\left(X_{\overline y}, \overline{\Delta|_{W_{\overline y}}}\right)$ is $F$-pure for every $y\in Y$, where $X_{\overline y}$ is the geometric fiber of $f$ over $y$ and $\overline{\Delta|_{W_{\overline y}}}$ is the $\mathbb Z_{(p)}$-AC divisor on $X_{\overline y}$ obtained as the unique extension of $\Delta|_{W_{\overline y}}$. 
\end{itemize}
Let $e_0\ge 1$ be the smallest integer such that 
$(p^{e_0}-1)(K_X+\Delta)$ is Cartier. 
Let $A$ be an $f$-ample Cartier divisor on $X$. 
Then there exists an integer $m_0\ge 1$ such that 
\begin{align*}
& {f_{Z}}_*\phi_{(X_Z, \Delta_Z)}^{(e)}(mA_Z +N_Z):
\\ & {F_Z^e}_*{f_Z}_*\mathcal O_{X_Z}(
(1-p^e)(K_{X_Z}+\Delta_Z) +p^e(mA_Z+N_Z) )
\to {f_Z}_*\mathcal O_{X_Z}(mA_Z+N_Z)
\end{align*}
is surjective for every flat morphism $\pi:Z\to Y$ from a regular variety $Z$, 
for each $m\ge m_0$, $e\ge1$ with $e_0|e$,  
and for every $f$-nef Cartier divisor $N$ on $X$. 
\end{prop}
\begin{proof}
Let us consider the following commutative diagram: 
$$
\xymatrix{
	(X_{Z})^{e} \ar[rd]^-{F_{X_Z}^e} \ar[d]_-{F_{X_Z/Z}^{(e)}} & \\
	X_{Z^e} \ar[r]_-{w^(e)} \ar[d]_-{f_{Z^e}} & X_Z \ar[d]^-{f_Z} \\ 
	Z^e \ar[r]_-{F_Z^e} & Z
}
$$
Since $\mathcal O_{X_Z} \to {F_{X_Z}^e}_*\mathcal O_{X_Z}$ decomposes into 
$$
\mathcal O_{X_Z} 
\to w^{(e)}_* \mathcal O_{X_{Z^e}}
\to {F_{X_Z}^e}_* \mathcal O_{X_Z}, 
$$
we see that $\phi_{(X_Z,\Delta_Z)}^{(e)}(M_Z)$, where $M:=mA+N$, 
decomposes into 
\begin{align*}
\phi_{(X_Z,\Delta_Z)}^{(e)}(M_Z): 
& {F_{X_Z}^e}_*\mathcal O_{X_Z}((1-p^e)(K_{X_Z}+\Delta) +p^eM_Z)
\\ & \xrightarrow{\alpha^{(e)}} 
{w^{(e)}}_* \big( {f_{Z^e}}^* \mathcal O_Z((1-p^e)K_Z) (M_{Z^e})\big)
\xrightarrow{\beta^{(e)}} \mathcal O_{X_Z}(M_Z). 
\end{align*}
Here, 
$
\alpha^{(e)}={w^{(e)}}_* \phi^{(e)}_{(X_Z/Z,\Delta_Z)}( (1-p^e){f_{Z^e}}^*K_Z+M_{Z^e}). 
$
Note that, since $f$ is flat, we have  
\begin{align*}
\mathcal Hom({w^{(e)}}_*\mathcal O_{X_{Z^e}}, \mathcal O_{X_Z}(M_Z))
& \cong \mathcal Hom({w^{(e)}}_*\mathcal O_{X_{Z^e}}, \mathcal O_{X_Z})(M_Z)
\\ & \cong \mathcal Hom({f_Z}^*{F_Z^e}_*\mathcal O_Z, {f_Z}^*\mathcal O_Z)(M_Z)
\\ & \cong ({f_Z}^*\mathcal Hom({F_Z^e}_*\mathcal O_Z, \mathcal O_Z) )(M_Z)
\\ & \cong \big( {f_Z}^*{F_Z^e}_* \mathcal O_Z((1-p^e)K_Z) \big) (M_Z)
\\ & \cong \big( {w^{(e)}}_* {f_{Z^e}}^* \mathcal O_Z((1-p^e)K_Z) \big) (M_Z)
\\ & \cong {w^{(e)}}_* \big( {f_{Z^e}}^* \mathcal O_Z((1-p^e)K_Z) (M_{Z^e})\big)
\end{align*}
by the Grothendieck duality. 
We prove that both ${f_Z}_*\alpha^{(e)}$ and ${f_Z}_*\beta^{(e)}$ are surjective. 
By the projection formula, we have 
\begin{align*}
{f_Z}_*\alpha^{(e)}
& ={f_Z}_*{w^{(e)}}_* \phi^{(e)}_{(X_Z/Z,\Delta_Z)}((1-p^e){f_{Z^e}}^*K_Z+M_{Z^e}) 
\\ & \cong {F_Z^e}_*{f_{Z^e}}_* 
\phi^{(e)}_{(X_Z/Z,\Delta_Z)}((1-p^e){f_{Z^e}}^*K_Z+M_{Z^e}) 
\\ & \cong {F_Z^e}_* \big( {f_{Z^e}}_*\phi^{(e)}_{(X_Z/Z,\Delta_Z)}(M_{Z^e}) \big) 
((1-p^e)K_Z), 
\end{align*}
so it is enough to show that ${f_{Z^e}}_* \phi^{(e)}_{(X_Z/Z,\Delta_Z)}(M_{Z^e})$ 
is surjective. Note that since $F_Z^e$ is finite, the push-forward of 
a surjective morphism is also surjective. 
Since each square of the commutative diagram 
$$
\xymatrix{
(X_Z)^e \ar[r]^-{(\pi_X)^{(e)}} \ar[d]_-{F_{X_Z/Z}^{(e)}} & X^e \ar[d]^-{F_{X/Y}^{(e)}}  \\
X_{Z^e} \ar[r]^-{\left(\pi^{(e)}\right)_X} \ar[d]_-{f_{Z^e}} & X_{Y^e} \ar[d]^-{f_{Y^e}} \\ 
Z^e \ar[r]_-{\pi^{(e)}} & Y^e 
}
$$
is cartesian and $\pi$ is flat, we have 
\begin{align*}
{f_{Z^e}}_*\phi^{(e)}_{(X_Z/Z,\Delta_Z)}(M_{Z^e})
& \cong {f_{Z^e}}_* {\left(\pi^{(e)}\right)_X}^* \phi^{(e)}_{(X/Y,\Delta)}(M_{Y^e})
\\ & \cong {\pi^{(e)}}^* {f_{Y^e}}_* \phi^{(e)}_{(X/Y,\Delta)}(M_{Y^e}). 
\end{align*}
By \cite[Lemma~3.7~(2)]{Eji19p}, we see that there is an integer $m_0\ge 1$ 
such that 
$$
\phi^{(e)}_{(X/Y,\Delta)}(mA_{Y^e}+N_{Y^e})
$$ 
is surjective for each $m\ge m_0$ and every $f$-nef Cartier divisor $N$ on $X$, 
and hence so is ${f_{Z^e}}_*\phi^{(e)}_{(X_Z/Z,\Delta_Z)}(M_{Z^e})$. 
Next, we show that ${f_Z}_*\beta^{(e)}$ is surjective. Since 
$$
{w^{(e)}}_* \big( {f_{Z^e}}^*\mathcal O_Z((1-p^e)K_Z)(M_{Z^e}) \big)
\cong \big({f_Z}^*\mathcal Hom({F_Z^e}_*\mathcal O_Z, \mathcal O_Z) \big)(M_Z) 
$$
by the isomorphisms mentioned above, we get the commutative diagram 
$$
\xymatrix{
{f_Z}_* \big({f_Z}^*
\mathcal Hom({F_Z^e}_*\mathcal O_Z, \mathcal O_Z) \big)(M_Z) 
\ar[r]^-{{f_Z}_*\beta^{(e)}} & {f_Z}_* \mathcal O_{X_Z}(M_Z) \\ 
\mathcal Hom({F_Z^e}_*\mathcal O_Z,\mathcal O_Z) \otimes 
{f_Z}_* \mathcal O_{X_Z}(M_Z) \ar[r] \ar[u]^-\cong 
& {f_Z}_* \mathcal O_{X_Z}(M_Z) \ar@{=}[u].} 
$$
Here, the left vertical morphism is an isomorphism by the projection formula. 
Therefore, it is enough to prove that 
$
\mathcal Hom({F_Z^e}_*\mathcal O_Z) \to \mathcal O_Z
$
is surjective. The morphism is obtained as the dual of the Frobenius
morphism $\mathcal O_Z\to {F_Z^e}_* \mathcal O_Z$, 
which locally splits since $Z$ is regular, 
so the claim follows. 
\end{proof}
\begin{thm} \label{thm:relative} \samepage
Let $\bar f:\bar X\to \bar Y$ be a surjective morphism 
from a projective equi-dimensional $k$-scheme $\bar X$ 
to a projective variety $\bar Y$. 
Let $Y$ be a normal dense open subset of $\bar Y$ and 
set $X:=\bar f^{-1}(Y)$ and $f:=\bar f|_X:X\to Y$. 
Let $\bar A$ be a sufficiently ample Cartier divisor on $\bar X$ 
and put $A:=\bar A|_X$. 
Assume that $X$ satisfies $S_2$ and $G_1$. 
Let $\Delta$ be an effective $\mathbb Z_{(p)}$-AC divisor on $X$. 
Let $V$ be a dense open subset of $Y$ and set $U:=f^{-1}(V)$. 
Let $W$ be the largest Gorenstein open subset of $U$. 
Suppose further that 
\begin{itemize}
\item $K_X+\Delta$ is $\mathbb Z_{(p)}$-Cartier, 
\item $K_U+\Delta|_U$ is nef over $V$, 
\item $V$ is regular, 
\item $g:=f|_U:U\to V$ is flat, 
\item the codimension of $(U\setminus W)|_{U_v}$ $($resp. $\mathrm{Supp}(\Delta)|_{U_v})$ in $U_v$ is at least two $($resp. one$)$ for every $v\in V$, and 
\item $\left(U_{\overline v}, \overline{\Delta|_{W_{\overline v}}}\right)$ is $F$-pure for every $v\in V$, where $W_{\overline v}$ is the geometric fiber of $W$ over $v$ and $\overline{\Delta|_{W_{\overline v}}}$ is the $\mathbb Z_{(p)}$-AC divisor on $U_{\overline v}$ obtained as the unique extension of $\Delta|_{W_{\overline v}}$. 
\end{itemize}
Let $\bar N$ be a Cartier divisor on $\bar X$ such that 
$\bar A+\bar N$ is ample and $\bar N|_U$ is nef over $V$. Put $N:=\bar N|_X$. 
Let $m\ge 1$ be an integer such that $m(K_X+\Delta)$ is Cartier. 
Then 
$$
f_*\mathcal O_X(m(K_X+\Delta)+A+N) \otimes \omega_Y^{[-m]}
$$
is weakly positive over $V$. 
Here, $\omega_Y^{[-m]}:=\mathcal Hom(\omega_Y^{\otimes m}, \mathcal O_Y)$. 
\end{thm}
We have the following commutative diagram whose each square is cartesian: 
$$
\xymatrix{
W \ar@{^(->}[r] & U \ar[d]_-g \ar@{^(->}[r] 
& X \ar[d]^-f \ar@{^(->}[r] 
& \bar X \ar[d]^-{\bar f} \\
& V \ar@{^(->}[r]
& Y \ar@{^(->}[r] 
& \bar Y. 
}
$$
\begin{proof}
Since the weak positivity of a coherent sheaf $\mathcal G$ on $Y$ is 
equivalent to that of $\mathcal G|_{Y'}$, where $Y'$ is an open subset of $Y$ 
with $\mathrm{codim}_Y(Y\setminus Y') \ge 2$, 
we may assume that $Y$ is regular by shrinking $Y$. 
Set $\mathcal G:=f_*\mathcal O_X(M(K_{X/Y}+\Delta)+A+N)$. 
Let $\bar{\mathcal L}$ be an ample line bundle on $\bar Y$ 
with $|\bar{\mathcal L}|$ free, 
put $\mathcal L:=\bar{\mathcal L}|_Y$ 
and $\mathcal H:=\omega_Y^m \otimes \mathcal L^{m(n+1)}$.  
Replacing $\bar{\mathcal L}$ by its power, 
we may assume that $\mathcal H$ is ample and $|\mathcal H|$ is free.
Take an integer $\alpha\ge 1$. 
We prove that there is a $\beta \ge 1$ such that 
$$
S^{\alpha\beta}(\mathcal G) \otimes \mathcal H^\beta
$$
is globally generated over $V$. 
Take $d\ge 1$ so that $p^d\ge\alpha+1$. 
We consider the morphism $f_{Y^d}:X_{Y^d}\to Y^d$. 
Since $\bar A$ is sufficiently ample, we see that 
\begin{itemize}
\item the morphism 
\begin{align*}
& {f_{Y^d}}_*\phi^{(e)}_{(X_{Y^d},\Delta_{Y^d})}
\left(m\left(K_{X_{Y^d}}+\Delta_{Y^d}\right)+(A+N)_{Y^d}\right):
\\ & {F_{Y^d}^e}_* {f_{Y^d}}_* \mathcal O_{X^d}\left(
((m-1)p^e+1)\left(K_{X_{Y^d}}+\Delta_{Y^d}\right) +p^e(A+N)_{Y^d}\right)
\\ & \to {f_{Y^d}}_* \mathcal O_{X_{Y^d}}
\left(m \left(K_{X_{Y^d}}+\Delta_{Y^d} \right) +(A+N)_{Y^d} \right)
\end{align*}
is surjective over $V^d\subseteq Y^d$ for each $e\ge 1$ such that 
$(p^e-1)(K_X+\Delta)$ is Cartier 
by Proposition~\ref{prop:e-surjective}, 
\item the invertible sheaf 
\begin{align*}
& \mathcal O_{X_{Y^d}}\left(
m\left( K_{X_{Y^d}}+\Delta_{Y^d}\right)+(A+N)_{Y^d} \right)
\\ & \cong \mathcal O_{X_{Y^d}}\left( m\left( K_{X_{Y^d}/Y^d}+\Delta_{Y^d}\right)+(A+N)_{Y^d} +m{f_{Y^d}}^*K_{Y^d} \right)
\\ & \cong \big( \mathcal O_X(m(K_{X/Y}+\Delta) +A+N ) \big)_{Y^d} 
\otimes {f_{Y^d}}^* \omega_{Y^d}^m
\end{align*}
is relatively free over $V^d$ by Lemma~\ref{lem:rel free}, and  
\item the sheaf $\mathcal G$ is locally free over $V$. 
This follows from Keeler's relative Fujita vanishing \cite[Theorem~1.5]{Kee03}, 
the cohomology and base change, and the fact that the Euler characteristic of flat coherent sheaf is constant on fiber. 
\end{itemize}
Furthermore, by \cite[Lemma~6.6]{Eji19d}, 
we see that $X_{Y^d}$ satisfies $S_2$ and $G_1$. 
Therefore, we can apply Theorem~\ref{thm:PS-type} to $f_{Y^d}:X_{Y^d}\to Y^d$, 
from which we obtain that 
\begin{align*}
& \big( {f_{Y^d}}_*\mathcal O_{X_{Y^d}}(m(K_{X_{Y^d}} +\Delta_{Y^d}) 
+(A+N)_{Y^d}) \big) 
\otimes \mathcal L^{m(n+1)}
\\ \cong & \big( {f_{Y^d}}_* 
\mathcal O_{X_{Y^d}}(m(K_{X_{Y^d}/Y^d} +\Delta_{Y^d}) +(A+N)_{Y^d}) \big) 
\otimes \omega_Y^{m} \otimes \mathcal L^{m(n+1)}
\\ \cong & \big( {F_Y^d}^* f_*\mathcal O_X(m(K_{X/Y}+\Delta) +A+N) \big) 
\otimes \mathcal H
\\ = & {F_Y^d}^*\mathcal G \otimes \mathcal H
\end{align*}
is globally generated over $V^d$, 
where the isomorphism in the third line follows from the flatness of $F_Y^d$. 
Put $\mathcal E:=({F_Y^d}_*\mathcal O_Y) \otimes ({F_Y^d}_*\mathcal O_Y)^\vee$. 
Let $l\ge 1$ be an integer such that $\mathcal E\otimes \mathcal H^l$ 
is globally generated. 
Since the sheaf 
$$
S^{\alpha l p^d}\left( {F_Y^d}^*\mathcal G \otimes \mathcal H \right) 
\cong \left({F_Y^d}^*S^{\alpha l p^d}(\mathcal G)\right) 
\otimes \mathcal H^{\alpha l p^d} 
\cong {F_Y^d}^*\left( S^{\alpha l p^d}(\mathcal G) \otimes \mathcal H^{\alpha l} 
\right) 
$$
is globally generated over $V^d$, we have the morphism 
$$
\bigoplus \mathcal O_{Y}
\to 
{F_Y^d}^*\left( S^{\alpha l p^d}(\mathcal G) \otimes \mathcal H^{\alpha l} 
\right) 
$$
that is surjective over $V^d$.  Applying 
$
\left( {F_Y^d}_*(?) \right) \otimes \left({F_Y^d}_*\mathcal O_Y\right)^\vee 
\otimes \mathcal H^l
$
to this morphism, we get the morphisms 
\begin{align*}
\bigoplus \mathcal E \otimes \mathcal H^l 
& \to \left( {F_Y^d}_*{F_Y^d}^*\left( 
S^{\alpha l p^d}(\mathcal G) \otimes \mathcal H^{\alpha l} \right) \right) 
\otimes \left({F_Y^d}_*\mathcal O_Y \right)^\vee \otimes \mathcal H^l
\\ & \cong S^{\alpha l p^d}(\mathcal G) \otimes \mathcal H^{\alpha l} 
\otimes \left({F_Y^d}_*\mathcal O_Y\right) 
\otimes \left({F_Y^d}_*\mathcal O_Y \right)^\vee \otimes \mathcal H^l
\\ & \to
S^{\alpha l p^d}(\mathcal G) \otimes \mathcal H^{(\alpha+1)l} 
\end{align*}
by the projection formula, which are surjective over $V$. 
Note that the last morphism follows from the evaluation map 
$
\left({F_Y^d}_*\mathcal O_Y \right) 
\otimes \left({F_Y^d}_*\mathcal O_Y\right)^\vee \to \mathcal O_Y,
$ 
which is surjective over $V$, as ${F_Y^d}_*\mathcal O_Y$ is locally free over $V$. 
Since $\mathcal E \otimes \mathcal H^l$ is globally generated, 
$S^{\alpha l p^d}(\mathcal G) \otimes \mathcal H^{(\alpha+1)l}$
is globally generated over $V$. 
Also, by $p^d\ge \alpha+1$ and the freeness of $|\mathcal H|$, 
we see that 
$S^{\alpha l p^d}(\mathcal G) \otimes \mathcal H^{lp^d}$
is globally generated over $V$. 
Putting $\beta:=lp^d$, we conclude the proof. 
\end{proof}
\begin{rem} \label{rem:relative}
In the above proof, when $Y$ is regular, 
we showed that for each $\alpha\ge1$, there is a $\beta\ge 1$
such that $S^{\alpha\beta}(\mathcal G)\otimes \mathcal H^\beta$ 
is globally generated over $V$. 
This is slightly stronger than the weak positivity of $\mathcal G$, 
since in the definition of the weak positivity, the global generation of 
$
(S^{\alpha\beta}(\mathcal G))^{\vee\vee}\otimes \mathcal H^\beta
$
over $V$ is employed. Note that the natural morphism 
$$
S^{\alpha\beta}(\mathcal G)\otimes \mathcal H^\beta
\to (S^{\alpha\beta}(\mathcal G))^{\vee\vee}\otimes \mathcal H^\beta
$$
is an isomorphism over $V$ as $\mathcal G$ is locally free over $V$. 
In the proof of Theorem~\ref{thm:relative2}, we use the global generation 
of $S^{\alpha\beta}(\mathcal G)\otimes \mathcal H^\beta$ over $V$. 
\end{rem}
When the geometric generic fiber has nef canonical divisor, 
by an argument similar to that of the proof of Theorem~\ref{thm:relative}, 
we can prove the following theorem:
\begin{thm} \label{thm:relative_generic} \samepage
Let $\bar f:\bar X\to \bar Y$, $f:X\to Y$, $\bar A$ and $A$ be as in 
Theorem~\ref{thm:relative}. 
Assume that $X$ satisfies $S_2$ and $G_1$. 
Let $\Delta$ be an effective $\mathbb Z_{(p)}$-AC divisor on $X$. 
Suppose that 
\begin{itemize}
\item $K_X+\Delta$ is $\mathbb Z_{(p)}$-Cartier, 
\item $(K_X+\Delta)|_{X_{\overline\eta}}$ is nef, 
	where $X_{\overline\eta}$ is the geometric generic fiber of $f$, and
\item $\left(X_{\overline \eta}, \Delta|_{X_{\overline \eta}}\right)$ is $F$-pure. 
\end{itemize}
Let $\bar N$ be a Cartier divisor on $\bar X$ such that 
$\bar A+\bar N$ is ample and $\bar N|_{X_{\overline\eta}}$ is nef. 
Put $N:=\bar N|_X$. 
Let $m\ge 1$ be an integer such that $m(K_X+\Delta)$ is Cartier. 
Then 
$$
f_*\mathcal O_X(m(K_X+\Delta)+A+N) \otimes \omega_Y^{[-m]}
$$
is pseudo-effective in the sense of Definition~\ref{defn:positivity}~$($4$)$. 
\end{thm}
In the remaining of this section, we prove the following theorem: 
\begin{thm} \label{thm:relative2}  \samepage
Let $\bar f:\bar X\to \bar Y$, $f:X\to Y$, $\bar A$, $A$, $g:U\to V$, 
$\bar N$ and $N$ be as in Theorem~\ref{thm:relative}. 
Assume that $X$ satisfies $S_2$ and $G_1$. 
Let $\Delta$ be an effective $\mathbb Z_{(p)}$-AC divisor on $X$. 
Let $W$ be the largest Gorenstein open subset of $X$. 
Suppose that 
\begin{itemize}
\item $K_X+\Delta$ is $\mathbb Z_{(p)}$-Cartier, 
\item $K_U+\Delta|_U$ is nef over $V$, 
\item $g:U\to V$ is flat, 
\item the codimension of $(U\setminus W)|_{U_v}$ $($resp. $\mathrm{Supp}(\Delta)|_{U_v})$ in $U_v$ is at least two $($resp. one$)$ for every $v\in V$,  
\item $\left(U_{\overline v}, \overline{\Delta|_{W_{\overline v}}}\right)$ is $F$-pure for every $v\in V$, where $W_{\overline v}$ is the geometric fiber of $W$ over $v$ and $\overline{\Delta|_{W_{\overline v}}}$ is the $\mathbb Z_{(p)}$-AC divisor on $U_{\overline v}$ obtained as the unique extension of $\Delta|_{W_{\overline v}}$. 
\item $Y$ is regular 
$($this assumption is stronger than that of Theorem~\ref{thm:relative}$)$. 
\end{itemize}
Let $m\ge 1$ be an integer such that $m(K_X+\Delta)$ is Cartier. 
Let $\bar{\mathcal L}$ and $\bar{\mathcal H}$ be ample line bundles on $\bar Y$ 
with $|\bar{\mathcal L}|$ free. 
Put $\mathcal L:=\bar{\mathcal L}|_Y$ and $\mathcal H:=\bar{\mathcal H}|_Y$. 
Then the sheaf 
$$
f_*\mathcal O_X(m(K_{X/Y}+\Delta) +A +N) 
\otimes \omega_Y \otimes \mathcal L^{\dim Y} \otimes \mathcal H
$$
is generated by its global sections over $V$. 
\end{thm}
\begin{proof}
Put $n:=\dim Y$. 
We first prove that, for any coherent sheaf $\mathcal F$ on $Y$, 
there is an integer $m_0\ge 1$ such that 
$$
\left( {F_Y^e}_*(\mathcal F\otimes \mathcal H^m) \right) \otimes \mathcal L^n
$$
is globally generated for each $m\ge m_0$ and $e\ge 1$. 
Let $\bar{\mathcal F}$ be a coherent sheaf on $\bar Y$ with 
$\bar{\mathcal F}\cong \mathcal F$. 
Since $F_{\bar Y}$ is finite, we have 
\begin{align*}
H^i\left( \bar Y, \left( {F_{\bar Y}^e}_* 
\left( \bar{\mathcal F} \otimes \bar{\mathcal H}^m \right) \right)
\otimes \bar{\mathcal L^{n-i}} \right)
& \cong H^i\left( \bar Y, {F_{\bar Y}^e}_* \left( \bar{\mathcal F} \otimes 
\bar{\mathcal H}^m\otimes \bar{\mathcal L}^{(n-i)p^e} \right) \right) 
\\ & \cong H^i\left( \bar Y, \bar{\mathcal F} \otimes 
\bar{\mathcal H}^m\otimes \bar{\mathcal L}^{(n-i)p^e} \right) 
\end{align*}
for each $0<i\le n$. Therefore, Fujita's vanishing theorem implies that 
there is $m_0\ge 1$ such that 
$\left( {F_{\bar Y}^e}_* \bar{\mathcal H}^m \right) \otimes \bar{\mathcal L}^n$ 
is $0$-regular with respect to $\bar{\mathcal L}$
in the sense of Castelnuovo--Mumford for each $m\ge m_0$ and $e\ge 1$, 
so it is globally generated by \cite[Theorem~1.8.5]{Laz04I}. 
Hence, 
$$
\left( {F_Y^e}_*(\mathcal F\otimes \mathcal H^m) \right) \otimes \mathcal L^n
\cong \left( \left({F_{\bar Y}^e}_*(\bar{\mathcal F} \otimes 
\bar{\mathcal H}^m \right) \otimes \bar{\mathcal L}^n \right)|_Y
$$
is also globally generated. 
Next, we show the assertion. 
Put 
$$
\mathcal G:=f_*\mathcal O_X(m(K_{X/Y}+\Delta) +A +N).
$$ 
Let $q_e$ and $r_e$ be the quotient and the remainder of the division of 
$(m-1)p^e+1$ by $m$. 
Then by the same argument as that of the proof of Theorem~\ref{thm:PS-type}, 
we obtain the morphism 
\begin{align*}
\bigoplus{F_Y^e}_*\big(S^{q_e}(f_*\mathcal O_X(m(K_X+\Delta)+A+N)) \big)
\to f_*\mathcal O_X(m(K_X+\Delta) +A +N)
\end{align*}
that is surjective over $V$. 
Note that $(U,\Delta|_U)$ is $F$-pure by the assumption. 
Taking the tensor product of the above morphism and 
$\omega_Y^{1-m} \otimes \mathcal L^n \otimes \mathcal H$, we get the morphism 
\begin{align*}
\bigoplus \left( {F_Y^e}_*\big(
S^{q_e}(\mathcal G) \otimes \omega_Y^{1-r_e} \otimes \mathcal H^{p^e}
\big) \right) \otimes \mathcal L^n
\to \mathcal G \otimes \omega_Y \otimes \mathcal L^n \otimes \mathcal H
\end{align*}
that is surjective over $V$. 
When $m=1$, we have $q_e=1$ and $r_e=0$ for each $e\ge1$. 
In this case, the sheaf 
$
\left( {F_Y^e}_* 
\left( \mathcal G \otimes \omega_Y \otimes \mathcal H^{p^e} \right)
\right) \otimes \mathcal L^n
$
is globally generated for $e\gg0$ by the first part of the proof, 
so the assertion follows. 
We consider the case when $m\ge 2$. 
In this case, we have $q_e\xrightarrow{e\to \infty} \infty$. 
By Theorem~\ref{thm:relative} and what mentioned in Remark~\ref{rem:relative}, 
for any $\alpha\ge 1$, there is a $\beta\ge 1$ such that 
$
S^{\alpha\beta}(\mathcal G)\otimes \mathcal H^\beta
$ 
is globally generated over $V$. 
Applying \cite[Lemma~2.14~b)]{Vie95} for $\alpha=1$, we see that 
$S^{q_e}(\mathcal G) \otimes \mathcal H^{q_e}$ is globally generated over $V$
for each $e\gg0$. 
Put $\mathcal E:=\bigoplus_{0\le r < m} \omega_Y^{1-r}$. 
Then we have the following sequence of morphisms that are surjective over $V$:
\begin{align*}
\bigoplus \bigoplus \left( {F_Y^e}_* 
\left( \mathcal E \otimes \mathcal H^{p^e-q_e} \right) \right) 
& \twoheadrightarrow 
\bigoplus \bigoplus \left( {F_Y^e}_* 
\left( \omega_Y^{1-r_e} \otimes \mathcal H^{p^e-q_e} \right) \right) 
\\ & \cong  
\bigoplus \left( {F_Y^e}_*\left( \left( \bigoplus \mathcal O_Y\right) 
\otimes \omega_Y^{1-r_e} \otimes \mathcal H^{p^e-q_e} \right) \right)
\\ & \to 
\bigoplus \left( {F_Y^e}_*\left( S^{q_e}(\mathcal G)\otimes \mathcal H^{q_e}
\otimes \omega_Y^{1-r_e} \otimes \mathcal H^{p^e-q_e} \right) \right)
\\ & \cong 
\bigoplus \left( {F_Y^e}_*\left( S^{q_e}(\mathcal G)
\otimes \omega_Y^{1-r_e} \otimes \mathcal H^{p^e} \right) \right)
\\ & \to 
\mathcal G \otimes \omega_Y \otimes \mathcal L^n \otimes \mathcal H. 
\end{align*}
Here, the morphism in the third line follows from the fact that 
$S^{q_e}(\mathcal G)\otimes \mathcal H^{q_e}$ is globally generated over $V$, 
and the last morphism was constructed in the above argument. 
Since $p^e-q_e \xrightarrow{e\to\infty}\infty$, 
the sheaf ${F_Y^e}_*(\mathcal E\otimes \mathcal H^{p^e-q_e})$ is 
globally generated for each $e\gg0$ by the first part of the proof, 
so the above sequence of morphisms implies that 
$\mathcal G \otimes \omega_Y \otimes \mathcal L^n \otimes \mathcal H$ 
is globally generated over $V$. 
\end{proof}
When the geometric generic fiber has nef canonical divisor, 
by an argument similar to the above, we can prove the following theorem: 
\begin{thm} \label{thm:relative2_generic}  \samepage
Let $\bar f:\bar X\to \bar Y$, $f:X\to Y$, $\bar A$ and $A$
be as in Theorem~\ref{thm:relative}. 
Assume that $X$ satisfies $S_2$ and $G_1$. 
Let $\Delta$ be an effective $\mathbb Z_{(p)}$-AC divisor on $X$. 
Suppose that 
\begin{itemize}
\item $K_X+\Delta$ is $\mathbb Z_{(p)}$-Cartier, 
\item $(K_X+\Delta)|_{X_{\overline \eta}}$ is nef, where $X_{\overline\eta}$ is the geometric generic fiber of $f$, 
\item $\left(X_{\overline\eta}, \Delta|_{X_{\overline\eta}}\right)$ is $F$-pure, and 
\item $Y$ is regular. 
\end{itemize}
Let $m\ge 1$ be an integer such that $m(K_X+\Delta)$ is Cartier. 
Let $\bar N$ be a Cartier divisor on $\bar X$ such that 
$\bar A +\bar N$ is ample and $\bar N|_{X_{\overline\eta}}$ is nef. 
Put $N:=\bar N|_X$. 
Let $\bar{\mathcal L}$ and $\bar{\mathcal H}$ be ample line bundles on $\bar Y$ 
with $|\bar{\mathcal L}|$ free. 
Put $\mathcal L:=\bar{\mathcal L}|_Y$ and $\mathcal H:=\bar{\mathcal H}|_Y$. 
Then the sheaf 
$$
f_*\mathcal O_X(m(K_{X/Y}+\Delta) +A +N) 
\otimes \omega_Y \otimes \mathcal L^{\dim Y} \otimes \mathcal H
$$
is generically generated by its global sections. 
\end{thm}
\section{Numerical Kodaira dimension of algebraic fiber spaces}
\label{section:subadditivity}
The main theorem of this section is the following: 
\begin{thm} \label{thm:subadditivity} \samepage
Let $X$ be a normal projective variety 
and let $\Delta$ be an effective $\mathbb Z_{(p)}$-Weil divisor on $X$. 
Let $Y$ be a regular projective variety 
and let $f:X\to Y$ be a surjective morphism. 
Suppose that 
\begin{itemize}
\item $K_X+\Delta$ is $\mathbb Z_{(p)}$-Cartier, 
\item $K_{X_{\overline\eta}}+\Delta|_{X_{\overline\eta}}$ is nef, 
where $X_{\overline\eta}$ is the geometric generic fiber of $f$, and 
\item $\left(X_{\overline\eta},\Delta|_{X_{\overline\eta}}\right)$ is $F$-pure. 
\end{itemize}
Let $D$ be a $\mathbb Q$-Cartier divisor on $X$ such that 
$D-(K_{X/Y}+\Delta)$ is nef. 
Then for every $\mathbb Q$-Cartier divisor $E$ on $Y$, 
$$
\kappa_\sigma(D+f^*E) 
\ge \kappa_\sigma\left(D|_{X_{\overline\eta}}\right) 
+\kappa(E)
$$
and 
$$
\kappa_\sigma(D+f^*E) 
\ge \kappa\left(D|_{X_{\overline\eta}}\right) 
+\kappa_\sigma(E). 
$$ 
In particular, 
$$
\kappa_\sigma(K_X+\Delta) 
\ge \kappa_\sigma\left(
K_{X_{\overline\eta}}+\Delta|_{X_{\overline\eta}}\right) 
+\kappa(Y)
$$
and 
$$
\kappa_\sigma(K_X+\Delta) 
\ge \kappa\left(
K_{X_{\overline\eta}}+\Delta|_{X_{\overline\eta}}\right) 
+\kappa_\sigma(Y). 
$$
\end{thm}
\begin{rem} \label{rem:normality}
The normality of $X$ and the $F$-purity of $X_{\overline\eta}$ 
induce the normality of $X_{\overline\eta}$. 
This follows from an argument similar to that of the proof of \cite[Proposition~5.7~(4)]{Eji19w}. 
\end{rem}
To prove Theorem~\ref{thm:subadditivity}, we use the following lemmas. 
\begin{lem}[\textup{\cite[Corollary~2.1.38]{Laz04I}}] \label{lem:bound}
Let $Y$ be a normal projective variety and let $D$ be a Cartier divisor on $Y$. 
Then there exists an integer $m_0 \ge 1$ 
and a real number $\gamma>0$ such that 
$$
\gamma m^{\kappa(D)} \le h^0(Y, mm_0D). 
$$
for each $m\gg0$. 
\end{lem}
\begin{lem} \label{lem:multiple}
Let $Y$ be a normal projective variety and let 
$D$ be a $\mathbb Q$-Cartier $\mathbb Q$-divisor on $Y$. 
Then for each $l\ge 1$, the equalities 
$$
\kappa_\sigma^{\mathbb R}(D) 
= \kappa_\sigma^{\mathbb R}(lD)
$$ 
and 
$$
\kappa_\sigma(D) = \kappa_\sigma(lD)
$$
holds. 
\end{lem}
\begin{proof}
We prove that 
$\kappa_\sigma^{\mathbb R}(D) \ge \kappa_\sigma^{\mathbb R}(lD)$. 
Assume that $\kappa_\sigma^{\mathbb R}(lD) \ge 0$. 
Take $\delta\in\mathbb R$ with $\delta<\kappa_\sigma^{\mathbb R}(lD)$. 
Then for a Cartier divisor $A$ on $Y$, we have 
$$
\underset{m\to\infty}{\mathrm{lim~sup}}~
\frac{h^0(Y, \lfloor lmD \rfloor +A)}{m^\delta}
=:\gamma >0,  
$$
so we get 
$$
\underset{m\to\infty}{\mathrm{lim~sup}}~
\frac{h^0(Y, \lfloor mD \rfloor +A)}{m^\delta}
\ge 
\underset{m\to\infty}{\mathrm{lim~sup}}~
\frac{h^0(Y, \lfloor lmD \rfloor +A)}{(lm)^\delta}
=\gamma l^{\delta}>0. 
$$
This means that $\kappa_\sigma^{\mathbb R}(D) \ge \delta$, 
and hence $\kappa_\sigma^{\mathbb R}(D) \ge \kappa_\sigma^{\mathbb R}(lD)$. 
Next, we show that 
$\kappa_\sigma^{\mathbb R}(D)\le\kappa_\sigma^{\mathbb R}(lD)$. 
Assume that $\kappa_\sigma^{\mathbb R}(D)\ge0$. 
Take $\delta\in\mathbb R$ with $\delta<\kappa_\sigma^{\mathbb R}(D)$. 
Then for a Cartier divisor $A$ on $Y$, we have 
$$
\underset{m\to\infty}{\mathrm{lim~sup}}~
\frac{h^0(Y, \lfloor mD \rfloor +A)}{m^\delta}
>0.  
$$
If $h^0(Y,\lfloor mD\rfloor +A)>0$, we get injective morphisms 
$$
\mathcal O_Y(\lfloor mD \rfloor +A)
\hookrightarrow \mathcal O_Y(l\lfloor mD \rfloor +lA)
\hookrightarrow \mathcal O_Y(\lfloor lmD \rfloor +lA), 
$$
so 
$$
\underset{m\to\infty}{\mathrm{lim~sup}}~
\frac{h^0(Y, \lfloor m(lD) \rfloor +lA)}{m^\delta}
\ge 
\underset{m\to\infty}{\mathrm{lim~sup}}~
\frac{h^0(Y, \lfloor mD \rfloor +A)}{m^\delta}
>0,  
$$
which means that $\delta<\kappa_\sigma^{\mathbb R}(lD)$, 
and hence $\kappa_\sigma^{\mathbb R}(D) \le \kappa_\sigma^{\mathbb R}(lD)$. 
The second equality in the assertion follows from an argument 
similar to the above. 
\end{proof}
\begin{proof}[Proof of Theorem~\ref{thm:subadditivity}]
We first prove the first inequality. 
We may assume that 
$
\kappa_\sigma^{\mathbb R}\left(D|_{X_{\overline\eta}}\right)
\ge 0. 
$
Let $i$ be a positive integer such that 
$iD$, $i(K_X+\Delta)$ and $iE$ are Cartier. 
By Lemma~\ref{lem:multiple}, we have 
$
\kappa_\sigma^{\mathbb R}\left(iD|_{X_{\overline\eta}}\right)
= \kappa_\sigma^{\mathbb R}\left(D|_{X_{\overline\eta}}\right). 
$
Take $\delta\in\mathbb R$ with 
$
\delta < \kappa_\sigma^{\mathbb R}\left(iD|_{X_{\overline\eta}}\right). 
$
Then we have  
$$
\underset{m\to\infty}{\mathrm{lim~sup}}~
\frac{h^0\left(X_{\overline\eta}, imD|_{X_{\overline\eta}} +B\right)}{m^\delta} 
>0 
$$
for some ample Cartier divisor $B$ on $X_{\overline\eta}$. 
Let $A'$ be a sufficiently ample Cartier divisor on $X$. 
Then $h^0\left(X_{\overline\eta}, A'|_{X_{\overline\eta}}-B\right)>0$. 
Let $L$ be an ample divisor on $Y$ with $|L|$ free. 
Put $n:=\dim Y$. 
For each $m\ge1$ with $i|m$, applying Theorem~\ref{thm:relative2_generic} for 
$N=mD-m(K_{X/Y}+\Delta)$, we obtain that 
$$
f_* \mathcal O_X(m(K_{X/Y}+\Delta)+A'+N) \otimes \mathcal O_Y((n+1)L)
\cong f_* \mathcal O_X(mD+A') \otimes \mathcal O_Y((n+1)L)
$$
is generically globally generated. 
Set $A:=A' +(n+1)f^*L$. 
Then $f_*\mathcal O_X(mD+A)$ is generically globally generated
for each $m\ge 1$ with $i|m$ by the projection formula. 
Put 
$
r_m:=h^0\left(X_{\overline\eta}, 
mD|_{X_{\overline\eta}}+A|_{X_{\overline\eta}}\right)
$ 
for each $m\ge 1$ with $i|m$. 
Then there is a generically isomorphic injective morphism 
$$
\bigoplus^{r_m} \mathcal O_Y 
\hookrightarrow f_*\mathcal O_Y(mD+A). 
$$
Taking the tensor product with $\mathcal O_Y(mE)$, we obtain 
$$
\bigoplus^{r_m} \mathcal O_Y(mE)
\hookrightarrow f_*\mathcal O_Y(m(D+f^*E)+A) 
$$
by the projection formula. From this, we get that 
$$
h^0\left(X_{\overline\eta}, 
mD|_{X_{\overline\eta}}+A|_{X_{\overline\eta}} \right)
\cdot h^0(Y,mE) \le h^0(X, m(D+f^*E) +A) 
$$
for each $m\ge 1$ with $i|m$. 
By Lemma~\ref{lem:bound}, there is a $\gamma'\in\mathbb R_{>0}$ and 
an integer $m_0\ge 1$ such that 
$$
\frac{h^0(Y, mm_0E)}{m^{\kappa(E)}} \ge \gamma' 
$$
for each $m\gg0$. Replacing $i$ by $im_0$, we obtain that 
$$
\frac{h^0(Y, imE)}{m^{\kappa(E)}} \ge \gamma 
$$
for some $\gamma\in\mathbb R_{>0}$ and each $m\gg0$. 
Then, we get that 
\begin{align*}
& \underset{m\to\infty}{\mathrm{lim~sup}}~
\frac{h^0(X,im(D+f^*E)+A)}{m^{\delta+\kappa(E)}}
\\ & \ge
\underset{m\to\infty}{\mathrm{lim~sup}}~
\frac{h^0\left(X_{\overline\eta}, 
imD|_{X_{\overline\eta}}+A|_{X_{\overline\eta}} \right)
\cdot h^0(Y,imE)}{m^{\delta+\kappa(E)}}
\\ & =
\underset{m\to\infty}{\mathrm{lim~sup}}~
\frac{h^0\left(X_{\overline\eta}, 
imD|_{X_{\overline\eta}}+A|_{X_{\overline\eta}} \right)}{m^{\delta}}
\cdot \frac{h^0(Y,imE)}{m^{\kappa(E)}}
\\ & \ge 
\underset{m\to\infty}{\mathrm{lim~sup}}~
\frac{h^0\left(X_{\overline\eta}, 
imD|_{X_{\overline\eta}}+A|_{X_{\overline\eta}} \right)}{m^{\delta}}
\cdot \gamma
\\ & \ge 
\underset{m\to\infty}{\mathrm{lim~sup}}~
\frac{h^0\left(X_{\overline\eta}, imD|_{X_{\overline\eta}}+B \right)}
{m^{\delta}}
\cdot \gamma
\\ & >0. 
\end{align*}
Note that the inequality in the fifth line follows from 
$h^0\left(X_{\overline\eta}, A|_{X_{\overline\eta}} -B\right)>0$, 
which is equivalent to 
$h^0\left(X_{\overline\eta}, A'|_{X_{\overline\eta}} -B\right)>0$. 
The above inequalities mean that 
$$
\kappa_\sigma^{\mathbb R}(i(D+f^*E))
\ge \delta+\kappa(E), 
$$ 
and so 
$$
\kappa_\sigma^{\mathbb R}(i(D+f^*E))
\ge \kappa_\sigma^{\mathbb R}\left(iD|_{X_{\overline\eta}}\right) 
+\kappa(E). 
$$
Hence, we see from Lemma~\ref{lem:multiple} again that 
\begin{align*}
\kappa_\sigma^{\mathbb R}(D+f^*E)
\ge 
\kappa_\sigma^{\mathbb R}\left(
D|_{X_{\overline\eta}}\right)
+\kappa(E). 
\end{align*}
By an argument similar to the above, we see that 
$$
\kappa_\sigma(D+f^*E) \ge 
\kappa_\sigma\left(D|_{X_{\overline\eta}} \right)+\kappa(E). 
$$

Next, we show the second inequality of the assertion. 
We may assume that $\kappa_\sigma^{\mathbb R}(E)\ge 0$. 
Take $\delta\in\mathbb R$ so that $\delta<\kappa_\sigma^{\mathbb R}(E)$. 
Then 
$$
\underset{m\to\infty}{\mathrm{lim~sup}}~
\frac{h^0\left(Y, mE +C\right)}{m^\delta} >0 
$$
for some ample Cartier divisor $C$. 
By the above argument and the projection formula, 
we have the generically isomorphic injective morphism 
$$
\bigoplus^{r_m} \mathcal O_Y(mE+C)
\hookrightarrow f_*\mathcal O_Y(m(D+f^*E) +A +f^*C)
$$
for each $m\ge1$ with $i|m$. From this, we get that 
\begin{align*}
& h^0\left(X_{\overline\eta}, mD|_{X_{\overline\eta}}\right) 
\cdot h^0(Y,mE+C)
\\ \le & 
h^0\left(X_{\overline\eta}, mD|_{X_{\overline\eta}}+A|_{X_{\overline\eta}}\right)
\cdot h^0(Y,mE+C)
\\ \le & 
h^0(X, m(D+f^*E)+A+f^*C). 
\end{align*}
Note that $h^0\left(X_{\overline\eta},A|_{X_{\overline\eta}}\right)>0$ 
by the choice of $A'$. 
Using the same argument as above, we see that 
$$
\kappa_\sigma^{\mathbb R}(D+f^*E)
\ge \kappa\left(D|_{X_{\overline\eta}}\right) 
+\kappa_\sigma^{\mathbb R}(E)
$$
and 
$$
\kappa_\sigma(D+f^*E)
\ge \kappa\left(D|_{X_{\overline\eta}}\right) 
+\kappa_\sigma(E). 
$$
\end{proof} 
Next, we consider the case when the geometric generic fiber is not $F$-pure. 
The following example shows that algebraic fiber spaces constructed in~\cite{CEKZ21} violate the inequality in Theorem~\ref{thm:subadditivity}. 
\begin{eg} \label{eg:CEKZ}
We use the same notation as in \cite[\S 2]{CEKZ21}. 
To show the existence of counterexamples to Theorem~\ref{thm:subadditivity}, 
it is enough to show that $\kappa_\sigma(K_{X^{(m)}}) =-\infty$ for $m>pl$ 
(cf. \cite[Theorem~2.4]{CEKZ21}). 
Note that $K_{X^{(m)}}$ is relatively nef over the base space. 
 
Fix an integer $m>pl$. 
Let $\mu$ be a positive integer such that $A':=\pi^*T+\mu f^*D'$ is ample. 
Take an integer $\nu\gg0$. 
Put $A:=\nu \sum_{i=1}^m \mathrm{pr}_i^* A'$. 
It is enough to show that 
$$
H^0(X^{(m)}, nK_{X^{(m)}} +A) =0
$$
for each $n\gg0$. Since 
$$
nK_{S/C} +\nu A' 
\sim nq\pi^*T -nf^*D' +\nu \pi^*T +\mu\nu f^*D'
= (nq+\nu) \pi^*T +(\mu\nu -n)f^*D', 
$$
we have 
\begin{align*}
f_*(\omega_{S/C}^n(\nu A')) 
& \cong 
g_*\left(\mathcal O_P((nq+\nu)T +(\mu\nu -n) g^*D') 
\otimes \bigoplus_{i=1}^{l-1}\mathcal O_P(-irT -ig^*D')
\right)
\\ & \cong 
\bigoplus_{i=1}^{l-1} g_* \mathcal O_P((nq+\nu -ir)T +(\mu\nu -n-i)g^*D')
\\ & \cong 
\bigoplus_{i=1}^{l-1} S^{nq+\nu-ir}(\mathcal E) \otimes \mathcal O_C((\mu\nu-n-i)D). 
\end{align*}
From this, we obtain that 
\begin{align*}
& f^{(m)}_* \omega_{X^{(m)}}^n(A)
\\ & \cong 
f^{(m)}_* \omega_{X^{(m)}/C}^n(A) \otimes \omega_C^n
\\ & \cong 
f^{(m)}_* \left( \bigotimes_{i=1}^m \mathrm{pr}_i^*(\omega_{S/C}^n(\nu A')) \right) \otimes \omega_C^n
\\ & \cong 
\left( \bigotimes_{i=1}^m f_* (\omega_{S/C}^n(\nu A') ) \right) \otimes \omega_C^n
\\ & \cong 
\left( \bigoplus_{1\le i_1,\ldots, i_m \le l-1} \bigotimes_{j=1}^m S^{nq+\nu -i_jr}(\mathcal E) \right) \otimes \mathcal O_C\left(\left(n(pl-m) +m\mu\nu -\sum_{j=1}^m i_j\right)D'\right). 
\end{align*}
Since $m>pl$, this sheaf is anti-ample when $n\gg0$. 
Note that $\mathcal E$ is anti-nef. 
Hence, 
$$
H^0(X^{(m)}, nK_{X^{(m)}}+A) 
=H^0\left(C, f^{(m)}_*\omega_{X^{(m)}}^n(A)\right)
=0. 
$$
\end{eg}
\section{Algebraic fiber spaces over varieties of general type}
\label{section:relatively_semi-ample}
In this section, we prove that Iitaka's conjecture holds true if the geometric generic fiber has mild singularities, if the canonical divisor of the total space is relatively semi-ample, and if the base space is of general type. 
\begin{thm} \label{thm:relatively_semi-ample}
Let $X$ be a normal projective variety and let $\Delta$ be an effective $\mathbb Z_{(p)}$-Weil divisor on $X$. 
Let $Y$ be a regular projective variety and let $f:X\to Y$ be a surjective morphism. 
Suppose that 
\begin{itemize}
\item $K_X+\Delta$ is $\mathbb Z_{(p)}$-Cartier, 
\item $\left(X_{\overline\eta},\Delta|_{X_{\overline\eta}}\right)$ is $F$-pure, 
\item $K_X+\Delta$ is relatively semi-ample over $Y$, and 
\item $Y$ is of general type. 
\end{itemize}
Then 
$$
\kappa(K_X+\Delta) \ge \kappa\left(K_{X_{\overline\eta}}+\Delta|_{X_{\overline\eta}}\right) +\kappa(Y).
$$
\end{thm}
\begin{proof}
Let $A'$ be a sufficiently ample Cartier divisor on $X$. 
Let $L$ be an ample divisor on $Y$ with $|L|$ free. 
Put $A:=A'+f^*(K_Y+(\dim Y+1)L)$. 
Then for each $m\in\mathbb Z_{>0}$ with $m(K_X+\Delta)$ is Cartier, 
\begin{align*}
& f_*\mathcal O_X(m(K_{X/Y}+\Delta) +A)
\\ \cong & f_*\mathcal O_X(m(K_{X/Y}+\Delta) +A') \otimes \omega_Y 
\otimes \mathcal O_Y((\dim Y+1)L)
\end{align*}
by the projection formula, and the sheaf is generically globally generated 
by Theorem~\ref{thm:relative2_generic}. 
Take an $i\in\mathbb Z_{>0}$ so that $i(K_{X/Y}+\Delta)$ is 
Cartier and relatively free over $Y$. 
Then we have a normal projective variety $Z$, 
morphisms $g:X\to Z$ and $h:Z\to Y$ with $g_*\mathcal O_X \cong \mathcal O_Z$ 
and $h\circ g =f$, and an $h$-ample Cartier divisor $D$ on $Z$ such that 
$g^*D\sim i(K_{X/Y}+\Delta)$. 
Let $H$ be an ample divisor on $Y$ such that $D+h^*H$ is ample. 
Take an $m_0\in\mathbb Z_{>0}$ so that 
$$
\mathcal Hom\big(g_*\mathcal O_X(A), \mathcal O_Z(m_0(D+h^*H)) \big)
\cong \left(g_*\mathcal O_X(A) \right)^* \otimes \mathcal O_Z(m_0(D+h^*H)) 
$$
has a non-zero global section $\varphi'$. 
Let $\mathcal I (\ne 0)$ be the image of 
$$
\varphi':g_*\mathcal O_X(A)\to \mathcal O_Z(m_0(D+h^*H)).
$$ 
Let $\varphi:g_*\mathcal O_X(A) \to \mathcal I$ be the induced morphism. 
Since $D$ is $h$-ample, there is an $m_1\in\mathbb Z_{>0}$ such that 
$$
f_* \mathcal O_X(im(K_{X/Y}+\Delta) +A)
\cong h_* \big( (g_*\mathcal O_X(A) ) (mD) \big)
\xrightarrow{h_* \left(\varphi \otimes \mathcal O_Z(mD)\right)}
h_* \mathcal I(mD)
$$
is surjective for each $m\ge m_1$.  
Since the source of the morphism is generically globally generated, 
we obtain the injective morphism 
\begin{align*}
\bigoplus^{r_m} \mathcal O_Y
\hookrightarrow & h_* \mathcal O_Z((m+m_0)D + m_0h^*H) 
\\ \cong & f_* \mathcal O_X(i(m+m_0)(K_{X/Y}+\Delta)) \otimes \mathcal O_Y(m_0H)
\end{align*}
for each $m\ge m_1$, where $r_m :=\mathrm{rank}\,h_*\mathcal I(mD)$.
Taking the tensor product with $\omega_Y^{i(m+m_0)}\otimes\mathcal O_Y(-m_0H)$, 
we get the injective morphism 
$$
\bigoplus^{r_m} \mathcal O_Y(i(m+m_0)K_Y -m_0H)
\hookrightarrow f_*\mathcal O_X(i(m+m_0)(K_X+\Delta)), 
$$
from which we obtain the inequality 
$$
r_m \cdot h^0(Y, i(m+m_0)K_Y -m_0H) 
\le h^0(X, i(m+m_0)(K_X+\Delta)) 
$$
for each $m\ge m_1$. 
Since $D$ is $h$-ample, there is a $C_1\in\mathbb R_{>0}$ such that 
$$
r_m =\mathrm{rank}\,h_*\mathcal I(mD) 
=h^0\left(Z_\eta, \mathcal I|_{Z_\eta}\left(mD|_{Z_\eta}\right) \right)
\ge C_1 m^{\dim Z_\eta} 
=C_1 m^{\kappa\left(K_{X_{\overline\eta}}+\Delta|_{X_{\overline\eta}}\right)}
$$
for each $m\gg0$. 
Also, since $K_Y$ is big, there is a $C_2\in\mathbb R_{>0}$ such that 
$$
h^0(Y,i(m+m_0)K_Y -m_0H)
\ge 
C_2 m^{\kappa(Y)} 
$$ 
for each $m\gg0$. Hence, we get that 
$$
\kappa(K_X+\Delta) \ge \kappa\left(K_{X_{\overline\eta}}+\Delta|_{X_{\overline\eta}}\right) +\kappa(Y).
$$
\end{proof}
\section{Algebraic fiber spaces over varieties of maximal Albanese dimension}
\label{section:mAd}
In this section, we treat algebraic fiber spaces over varieties of maximal Albanese dimension. In this case, we assume that the total space is $F$-pure over a dense open subset of the base space, which is weaker than the assumption that the geometric generic fiber is $F$-pure. 
\begin{defn} \label{defn:mAd}
Let $Y$ be a projective variety. 
We say that $Y$ is of \textit{maximal Albanese dimension} if there exists 
a morphism $\pi:Y\to A$ to an abelian variety $A$ that is 
generically finite over its image. 
\end{defn}
\begin{thm} \label{thm:relative_mAd} \samepage
Let $\bar f:\bar X \to \bar Y$ be a surjective morphism from 
an equi-dimensional projective $k$-scheme $\bar X$ 
to a projective variety $\bar Y$ of maximal Albanese dimension. 
Let $\bar \pi:\bar Y\to B$ be a morphism to an abelian variety $B$ 
that is generically finite over its image. 
Let $B_0\subseteq B$ be a dense open subset such that 
$Y:=\bar\pi^{-1}(B_0) \ne\emptyset$. 
Put $X:=\bar f^{-1}(Y)$ and $f:=\bar f|_X:X\to Y$. 
Let $\bar A$ be a sufficiently ample Cartier divisor on $\bar X$ 
and set $A:=\bar A|_X$. 
Assume that $X$ satisfies $S_2$ and $G_1$. 
Let $\Delta$ be an effective $\mathbb Z_{(p)}$-AC divisor on $X$. 
Suppose further that 
\begin{itemize}
\item $K_X+\Delta$ is $\mathbb Z_{(p)}$-Cartier, 
\item $(X,\Delta)$ is $F$-pure over a dense open subset of $Y$, 
\item $(K_X+\Delta)|_{X_\eta}$ is nef, where $X_\eta$ is the generic fiber of $f$, and 
\item $Y$ is normal. 
\end{itemize}
Let $\bar N$ be a Cartier divisor on $\bar X$ such that $\bar A+\bar N$ is ample 
and $\bar N|_{X_\eta}$ is nef $($e.g. $N$ is nef$)$. 
Set $N:=\bar N|_X$. 
Let $m\ge1$ be an integer such that $m(K_X+\Delta)$ is Cartier. 
Then the sheaf 
$$
f_*\mathcal O_X(m(K_X+\Delta)+A+N)
$$
is pseudo-effective in the sense of Definition~\ref{defn:positivity}. 
\end{thm}
\begin{rem} \label{rem:nonnormal}
The reason why $Y$ is assumed to be normal is that the pseudo-effectivity is defined over normal varieties. In the case when $Y$ is not necessarily normal, we can prove that 
$$
\mathbb B_-\big(f_*\mathcal O_X(m(K_X+\Delta)+N+A) \big)
\ne Y. 
$$
For the definition of $\mathbb B_-$ for coherent sheaves, 
see \cite[\S 4]{Eji19d}. 
\end{rem}
\begin{proof}
Set $\pi:=\pi|_Y:Y\to B$. 
Let $\mathcal L$ be a symmetric ample line bundle on $B$ 
(i.e., $(-1)_B^*\mathcal L\cong \mathcal L$) with $|\mathcal L|$ free. 
Put 
$$
\mathcal H:=\mathcal L^{m(n+1)} 
\quad \textup{and} \quad 
\mathcal G:=f_*\mathcal O_X(m(K_X+\Delta)+A+N). 
$$ 
Take an integer $\alpha\ge1$.  
We show that there is an integer $\beta\ge1$ such that 
$$
S^{\alpha\beta}(\mathcal G) \otimes (\pi^*\mathcal H)^\beta
$$
is generically globally generated. 
If this holds, then the assertion follows from \cite[Lemma~2.14~a)]{Vie95}.
Since $\pi$ is generically finite, 
the natural morphism $\pi^*\pi_* \mathcal G \to \mathcal G$
is generically surjective, so the induced morphism 
$$
\pi^*\left( S^{\alpha\beta}(\pi_*\mathcal G) \otimes \mathcal H^\beta \right)
\cong 
S^{\alpha\beta}\left(\pi^*\pi_*\mathcal G \right) \otimes (\pi^*\mathcal H)^\beta
\to 
S^{\alpha\beta}(\mathcal G) \otimes (\pi^*\mathcal H)^\beta
$$
is also generically surjective.
Therefore, replacing $\bar f:\bar X\to \bar Y$ and $Y$ with 
$\bar\pi\circ \bar f:\bar X\to B$ and $B_0$, respectively,
we may assume that $\bar Y=B$ (with loss of surjectivity).
Let $t\ge 2$ be an integer with $p\nmid t$. 
Let $\sigma:B\to B$ be a morphism sending $b\in B$ to $t\cdot b\in B$. 
Then ${\sigma^e}^*\mathcal L \cong \mathcal L^{t^{2e}}$ 
and ${\sigma^e}^*\mathcal H \cong \mathcal H^{t^{2e}}$ 
for each $e\ge 1$. 
Fix an integer $d\ge 1$ with $t^{2d}\ge \alpha+1$. 
Let $B^{(d)}$ denote $B$ when we regard it as the source of $\sigma^d:B\to B$. 
Now, we have the following cartesian diagram:
$$
\xymatrix{
X_{B^{(d)}} \ar[d]_-{f_{B^{(d)}}} \ar[r]^-{(\sigma^d)_X} & X \ar[d]^-f \\
B^{(d)} \ar[r]_-{\sigma^d}& B
}
$$
Since $\sigma^d$ is \'etale, we see that 
\begin{itemize}
\item the morphism 
\begin{align*}
& {f_{B^{(d)}}}_*\phi^{(e)}_{\left(X_{B^{(d)}},\Delta_{B^{(d)}}\right)}
\left(m\left( K_{X_{B^{(e)}}}+\Delta_{B^{(d)}}\right)+ (A+N)_{B^{(e)}} \right):
\\ & {F_{B^{(d)}}^e}_* {f_{B^{(d)}}}_* \left( ((m-1)p^e+1)
\left(K_{X_{B^{(d)}}}+\Delta_{B^{(d)}}\right) 
+p^e(A+N)_{B^{(d)}} \right)
\\ & \to {f_{B^{(d)}}}_* \mathcal O_{X_{B^{(d)}}}\left( 
m\left( K_{X_{B^{(d)}}}+\Delta_{B^{(d)}}\right) +(A+N)_{B^{(d)}} \right)
\end{align*}
is generically surjective for each $e\ge 1$ such that 
$(p^e-1)(K_X+\Delta)$ is Cartier, and  
\item the invertible sheaf 
\begin{align*}
\mathcal O_{X_{B^{(d)}}}\left(m \left(K_{X_{B^{(d)}}}+\Delta_{B^{(d)}}\right)
+(A+N)_{B^{(d)}} \right)
\cong \big( \mathcal O_X(m(K_X+\Delta)+A+N) \big)_{B^{(d)}}
\end{align*}
is free on the generic fiber of $f$ (over its image). 
Note that $K_{X_{B^{(d)}}}$ is the pullback of $K_X$, 
since $X_{B^{(d)}}\to X$ is \'etale. 
\end{itemize}
Therefore, we can apply Theorem~\ref{thm:PS-type_generic} to 
$f_{B^{(d)}}:X_{B^{(d)}}\to {B^{(d)}}$ 
(this morphism is not necessarily surjective, but it is not an issue), 
from which we obtain that 
\begin{align*}
& {f_{B^{(d)}}}_*\mathcal O_{X_{B^{(d)}}}\left(m\left(K_{X_{B^{(d)}}}+\Delta_{B^{(d)}}\right) +(A+N)_{B^{(d)}}\right) \otimes \mathcal H
\\ & \cong 
{f_{B^{(d)}}}_*\big(\mathcal O_X\left(m\left(K_X+\Delta\right) +A+N\right)\big)_{B^{(d)}} \otimes \mathcal H
\\ & \cong 
{\sigma^d}^* \mathcal G \otimes \mathcal H
\end{align*}
is generically globally generated. 
Let $l\ge 1$ be an integer such that 
$$
(\sigma^d_*\mathcal O_B)
\otimes (\sigma^d_*\mathcal O_B) ^\vee
\otimes \mathcal H^l
$$
is globally generated. 
Then by an argument similar to that of the proof of Theorem~\ref{thm:relative}, 
we see that 
$
S^{\alpha l t^{2d}}(\mathcal G) \otimes \mathcal H^{(\alpha+1)l}
$
is generically globally generated. 
Since $t^{2d}\ge \alpha+1$ and $|\mathcal H|$ is free, 
$
S^{\alpha l t^{2d}}(\mathcal G) \otimes \mathcal H^{lt^{2d}}
$
is generically globally generated. 
Putting $\beta:=lt^{2d}$, we conclude the proof. 
\end{proof}
\begin{thm} \label{thm:relative2_mAd} \samepage
Let $\bar f:\bar X\to \bar Y$, $\bar\pi:\bar Y\to B$, $B_0$, $f:X\to Y$, 
$\bar A$ and $A$ be as in Theorem~\ref{thm:relative_mAd}. 
Assume that $X$ satisfies $S_2$ and $G_1$. 
Let $\Delta$ be an effective $\mathbb Z_{(p)}$-AC divisor on $X$. 
Suppose that 
\begin{itemize}
\item $K_X+\Delta$ is $\mathbb Z_{(p)}$-Cartier, 
\item $(X,\Delta)$ is $F$-pure over a dense open subset of $Y$, and 
\item $(K_X+\Delta)|_{X_\eta}$ is nef, where $X_\eta$ is the generic fiber of $f$. 
\end{itemize}
Let $\bar N$ be a Cartier divisor on $\bar X$ such that 
$\bar A+\bar N$ is ample and $N|_{X_\eta}$ is nef. 
Set $N:=\bar N|_X$. 
Let $\bar{\mathcal L}$ and $\bar{\mathcal H}$ be ample line bundles on $\bar Y$ 
with $|\bar{\mathcal L}|$ free, and 
put $\mathcal L:=\bar{\mathcal L}|_Y$ and $\mathcal H:=\bar{\mathcal H}|_Y$. 
Let $m\ge 1$ be an integer such that $m(K_X+\Delta)$ is Cartier. 
Then the sheaf 
$$
f_*\mathcal O_X(m(K_X+\Delta) +A +N) 
\otimes \mathcal L^{\dim Y} \otimes \mathcal H
$$
is generically generated by its global sections. 
\end{thm}
\begin{rem} \label{rem:nonnormal2}
In this theorem, we do not need to assume that $Y$ is normal.
\end{rem}
\begin{proof}
Set $n:=\dim Y$. Put 
$$
\mathcal G:= f_*\mathcal O_X(m(K_X+\Delta) +A +N). 
$$
Let $q_e$ and $r_e$ be the quotient and the remainder of the division of 
$(m-1)p^e+1$ by $m$. 
Then by the same argument as that of the proof of Theorem~\ref{thm:PS-type}, 
we get the generically surjective morphism
$$
\bigoplus {F_Y^e}_* S^{q_e}(\mathcal G) \to \mathcal G
$$
for each $e\ge1$ with $(p^e-1)(K_X+\Delta)$ is Cartier. 
Taking the tensor product with $\mathcal L^n \otimes \mathcal H$, we obtain
the generically surjective morphism 
$$
\bigoplus \left( {F_Y^e}_*\left(S^{q_e}(\mathcal G) 
\otimes \mathcal H^{p^e}\right) \right) \otimes \mathcal L^n
\to \mathcal G \otimes \mathcal L^n \otimes \mathcal H. 
$$
From here, we apply an argument similar to that of 
the proof of Theorem~\ref{thm:relative2}. 
Note that in the proof of Theorem~\ref{thm:relative_mAd}, we proved that 
for any integer $\alpha\ge 1$, there is an integer $\beta\ge1$ such that 
$
S^{\alpha\beta}(\mathcal G) \otimes \mathcal H^{\beta}
$
is generically globally generated. 
This $\mathcal H$ is different from that in this proof, 
but it is not an issue by \cite[Lemma~2.14~a)]{Vie95}. 
We combine the above generic global generation 
with the proof of Theorem~\ref{thm:relative2}. 
\end{proof}
\begin{thm} \label{thm:subadditivity_mAd} \samepage
Let $X$ be a normal projective variety
and let $\Delta$ be an effective $\mathbb Z_{(p)}$-Weil divisor on $X$. 
Let $Y$ be a normal projective variety 
and let $f:X\to Y$ be a surjective morphism. 
Suppose that 
\begin{itemize}
\item $K_X+\Delta$ is $\mathbb Z_{(p)}$-Cartier, 
\item $K_{X_\eta}+\Delta|_{X_\eta}$ is nef, 
where $X_\eta$ is the generic fiber of $f$, and 
\item $(X,\Delta)$ is $F$-pure over a dense open subset of $Y$. 
\end{itemize}
Let $D$ be a $\mathbb Q$-Cartier divisor on $X$ such that 
$D-(K_X+\Delta)$ is nef. 
Then for every $\mathbb Q$-Cartier divisor $E$ on $Y$, 
$$
\kappa_\sigma(D+f^*E) 
\ge \kappa_\sigma\left(D|_{X_\eta}\right)
+\kappa(E)
$$
and 
$$
\kappa_\sigma(D+f^*E) 
\ge \kappa\left(D|_{X_\eta}\right)
+\kappa_\sigma(E)
$$
In particular, 
$$
\kappa_\sigma(K_X+\Delta) 
\ge \kappa_\sigma\left(K_{X_\eta}+\Delta_{X_\eta}\right). 
$$
\end{thm}
\begin{proof}
This theorem is proved by an argument similar to that of the proof of 
Theorem~\ref{thm:subadditivity}, by using Theorem~\ref{thm:relative2_mAd} 
instead of Theorem~\ref{thm:relative2_generic}. 
\end{proof}
\bibliographystyle{abbrv}
\bibliography{ref.bib}
\end{document}